\documentclass[12pt,a4paper,oneside]{article}
\usepackage{amsmath}
\usepackage{amsfonts}
\usepackage{amsthm}
\usepackage[symbol]{footmisc}
\usepackage{csquotes}
\usepackage{mathrsfs}
\usepackage{dsfont}
\usepackage[colorlinks=true,linkcolor=blue]{hyperref}
\newcommand{\PP}{{\text{\usefont{U}{dsrom}{m}{n}P}}}
\newcommand{\Pas}{$\PP$-a.s in $\Omega$}
\newcommand{\HS}{\text{HS}(L^2(D))}
\newcommand{\eps}{\varepsilon}
\newcommand{\pe}{\psi_{\eps}}
\newcommand{\di}{\displaystyle}

\newcommand{\R}{\mathbb{R}}
\newcommand{\na}{\mathbb{N}}
\newcommand{\plap}{\Delta_p^{\mathcal{N}}}
\newcommand{\phe}{\phi_{\eps}}

\newcommand{\ups}{u_{\eps}}
\newcommand{\upsk}{\ups^k}
\newcommand{\upsn}{\ups^n}
\newcommand{\upsz}{\ups^0}
\newcommand{\upskp}{\ups^{k+1}}
\newcommand{\upsnp}{\ups^{n+1}}
\newcommand{\uhnr}{u_{N,\eps}}
\newcommand{\uhnl}{\uhnr^{\mathcal{P}}}

\newcommand{\diver}{\operatorname{div}_x}
\newcommand{\dvz}{\diver^{\mathcal{N}}{\vec{z}}}
\newcommand{\erw}{\mathbb{E}}
\newcommand{\erwb}{\mathbb{E}\left[}
\newcommand{\erwe}{\right]}
\newcommand{\erww}[1]{\mathbb{E}\left[{#1}\right]}
\newtheorem{thm}{Theorem}[section]
\newtheorem{lem}[thm]{Lemma}
\newtheorem{prop}[thm]{Proposition}
\theoremstyle{definition}
\newtheorem{defn}[thm]{Definition}
\theoremstyle{remark}
\newtheorem{rem}[thm]{Remark}
\newtheorem*{ex}{Example}
\numberwithin{equation}{section}
\title{The Neumann problem for a multivalued $p$-Laplace equation of Allen-Cahn type with a multiplicative stochastic force}
\author{Caroline Bauzet\footnotemark[1], \and Fr\'ed\'eric Lebon\footnotemark[1], \and Kerstin Schmitz\footnotemark[2], \and C\'edric Sultan\footnotemark[1], \and Aleksandra Zimmermann\footnotemark[2]}
\date{\today}
\begin{document}
\maketitle
\begin{abstract} 
In this paper, we consider a parabolic problem with constraint written as a differential inclusion, driven by a multiplicative colored noise and involving a $p$-Laplace operator (for $p\geq 2$), nonlinear random source terms and subject to Neumann boundary conditions on a bounded Lipschitz domain of $\mathbb{R}^d$ with $d\geq 1$. This contribution aims at proving existence and uniqueness of a solution for such a multivalued problem.
On one hand, the existence result is proved by the analysis of a semi-implicit time discretization scheme
constructed on a smoother version of our problem, itself obtained by a regularization \enquote{\`a la Moreau-Yosida} of the subdifferential term. The key point of our approach consists in finding a clever relation between the time step denoted $\tau$ and the Moreau-Yosida regularization parameter denoted $\eps$ in view to pass simultaneously to the limit with respect to $\tau$ and $\eps$. On the other hand, the uniqueness of the solution is proved by standard arguments.
\end{abstract}

\footnotetext[1]{Aix Marseille Univ, CNRS, Centrale Med, LMA, Marseille, France, caroline.bauzet@univ-amu.fr, frederic.lebon@univ-amu.fr, cedric.sultan@univ-amu.fr}
\footnotetext[2]{TU Clausthal, Institut f\"ur Mathematik, Clausthal-Zellerfeld, Germany, kerstin.schmitz@tu-clausthal.de, 
aleksandra.zimmermann@tu-clausthal.de}
\section{Introduction}
\subsection{Statement of the problem}
We consider $T>0$, $D$ a bounded Lipschitz domain of $\R^d$ (with $d\in\na^\star$) and a stochastic basis $(\Omega,\mathcal{F},(\mathcal{F}_t)_{t\geq 0},\mathds{P})$ with on the filtration the usual assumptions, \textit{i.e.,} $\mathcal{F}_0$ contains all negligible sets of $\mathcal{F}$ and $(\mathcal{F}_t)_{t\geq 0}$ is right-continuous.
We are interested in the following multivalued $p$-Laplace equation of Allen-Cahn type forced by a multiplicative stochastic noise, with a given initial condition $u_0$ and homogeneous Neumann boundary conditions (denoting by $\mathbf{n}$ the unit normal vector to $\partial D$ outward to $D$):
\begin{align}\label{equation0}
\begin{aligned}
\beta(u)+f - \plap u-\partial_t\left(u-\int_0^\cdot G(u)\,dW(t)\right) &\in \partial I_{[0,1]}(u), &&\text{ in }\Omega\times(0,T) \times D;\\
u(0,\cdot)&=u_0, &&\text{ in } \Omega\times D;\\
\nabla u\cdot \mathbf{n}&=0, &&\text{ on }\Omega\times(0,T)\times\partial D.
\end{aligned}
\end{align}
In our setting, $p\geq 2$ and $\plap$ is the $p$-Laplace operator associated to the weak homogeneous Neumann boundary conditions defined on the Sobolev space $V:=W^{1,p}(D)$ and with values in the corresponding dual space $V^*$. It is defined by the following duality product: for all $u,v \in V$
\begin{align*}
\left\langle \plap u,v\right\rangle_{V^*,V}
=\int_D \left(|\nabla u(x)|^{p-2} \nabla u(x) \cdot \nabla v(x) + |u(x)|^{p-2}u(x) v(x)\right)\,dx.
\end{align*}
The Hilbert space valued Wiener process $(W(t))_{t\geq 0}$ is defined as follows. Firstly, we denote by $(\mathcal{W}^j)_{j\in\na^\star} = ((\mathcal{W}^j_t)_{t\geq 0})_{j\in\na^\star}$ a sequence of real-valued independent Wiener processes with respect to $(\mathcal{F}_t)_{t\geq 0}$. Secondly, a separable Hilbert space $U$ containing $L^2(D)$ is chosen, as well as a non-negative, symmetric trace class operator $\mathcal{Q}:U\to U$ with $\mathcal{Q}^{1/2}(U)=L^2(D)$. Thirdly, we choose an orthonormal basis $(e_j)_j$ of $U$ consisting in $\mathcal{Q}$ eigenvectors, matched to eigenvalues $(\lambda_j)_j\subset [0,\infty)$. Then, according to \cite[Section 4.1.1]{DPZ14}, the stochastic process $(W(t))_{t\geq 0}$ defined for any $t\geq 0$ by
\begin{equation*}
W(t):=\sum_{j=1}^{\infty}\sqrt{\lambda_j} e_j\mathcal{W}^j_t=\sum_{j=1}^{\infty}\mathcal{Q}^{1/2}(e_j)\mathcal{W}^j_t,
\end{equation*}
is a $(\mathcal{F}_t)$-adapted $\mathcal{Q}$-Wiener process with values in $U$.\\
In the sequel, the product space $\Omega \times (0,T)$ will be denoted by $\Omega_T$, whilst the predictable $\sigma$-field on $\Omega_T$ will be denoted by $\mathcal{P}_T$.\footnote[3]{$\mathcal{P}_{T}:=\sigma(\{ (s,t]\times F_s \ | \ 0\leq s < t \leq T, \ F_s\in \mathcal{F}_s \} \cup \{\{0\}\times F_0 \ | \ F_0\in \mathcal{F}_0 \})$ (see \cite[p. 33]{LR}). Note that a map to a separable Banach space $E$ from $\Omega_T$ is predictable when it is $\mathcal{P}_{T}$-measurable.}
For any separable Banach space $X$ and $1\leq q<\infty$, the subset of $L^q(\Omega_T;X)$ containing predictable processes will be denoted by
\begin{equation*}
L^q_{\mathcal{P}_T}(\Omega_T;X):=\left\{v\in L^q(\Omega_T;X) \text{ such that } v\text{ is } \mathcal{P}_T\text{-measurable} \right\}.
\end{equation*}
The measurability of $v\in L^q_{\mathcal{P}_T}(\Omega_T;X)$ is here understood as the predictable measurability and in particular $v$ is a random variable with values in $L^q(0,T;X)$.\\
Denoting the separable Hilbert space of Hilbert-Schmidt operators from $L^2(D)$ to $L^2(D)$ by $\HS$, the stochastic It\^{o} integral of a mapping $\mathscr{H} \in L^2_{\mathcal{P}_T}(\Omega_T;\HS)$ with respect to $(W(t))_{t\geq 0}$ may be defined as in \cite[Section 4.2]{DPZ14} and will be denoted by
\begin{equation*}
\int_0^t \mathscr{H}(s)\,dW(s), \ t \in [0,T].
\end{equation*}
The indicator function $I_{[0,1]}:\R\to\R\cup
\{+\infty\}$ being defined as
\begin{align*}
I_{[0,1]}(r)=\begin{cases} 0 & \text{if} \ r\in [0,1]\\
+\infty & \text{else}
\end{cases},
\end{align*}
its sub-differential is the following multivalued monotone operator $\partial I_{[0,1]}:[0,1]\to\mathcal{P}(\R)$:
\begin{align*}
\partial I_{[0,1]}(r)=
\begin{cases}
(-\infty,0] & \text{if} \ r=0\\
\{0\} & \text{if} \ r\in (0,1)\\
[0,\infty)& \text{if} \ r=1.
\end{cases}
\end{align*}
Note that Problem \eqref{equation0} takes a differential inclusion form and can be rewritten using Lagrange multipliers. Finding a solution $u$ for Problem \eqref{equation0} is equivalent to looking for a pair $(u,\psi)$ satisfying the following system written in a probabilistic form:
\begin{align}\label{equation}
\begin{aligned}
du+(\plap u+\psi)\,dt &=G(u)\,dW(t)+\left(\beta(u)+f\right)\,dt, &&\text{ in }\Omega\times(0,T) \times D;\\
\psi&\in \partial I_{[0,1]}(u)&&\text{ in }\Omega\times(0,T) \times D;\\
u(0,\cdot)&=u_0, &&\text{ in } \Omega\times D;\\
\nabla u\cdot \mathbf{n}&=0, &&\text{ on }\Omega\times(0,T)\times\partial D.
\end{aligned}
\end{align}
We have the following assumptions on the data:
\begin{itemize}
\item[$(H_1)$:] $u_0\in L^2\left(\Omega;L^2(D)\right)$ is such that $0\leq u_0(\omega,x) \leq 1$, for almost all $(\omega,x)\in \Omega\times D$. Moreover, $u_0$ is $\mathcal{F}_0$-measurable.
\item[$(H_2)$:] $\beta: \R\to\R$ is a $L_\beta$-Lipschitz continuous function (with $L_\beta > 0$) such that, for convenience, $\beta(0)=0$. 
\item[$(H_3)$:] $f\in L^2_{\mathcal{P}_T}(\Omega_T;L^2(D))$.
\end{itemize}
\begin{itemize}
\item[$(H_4)$:] For any $j\in\na^\star$ there exists a Lipschitz continuous function $g_j:\R\to\R$ such that $\operatorname{supp}(g_j)\subset [0,1]$
and the diffusion coefficient $G:L^2(D)\to\HS$ satisfies
\begin{equation*}
G(v)\left(\mathcal{Q}^{1/2}(e_j)\right)=g_j(v), \ \forall v\in L^2(D).
\end{equation*}
\item[$(H_5)$:] There exists a constant $K_g\geq 0$ such that
\begin{equation*}
\sup_{j\in\na^\star} \|g_j\|_{\infty} \leq K_g.
\end{equation*}
\item[$(H_6)$:] There exist constants $r \in [1,2)$ and $L_{G,r}>0$ such that
\begin{equation*}
\left(\sum_{j=1}^\infty \|g_j'\|_{\infty}^r\right)^{\frac{1}{r}} \leq L_{G,r}.
\end{equation*}
\end{itemize}
Let us mention that using Assumption $(H_4)$, the stochastic It\^o integral of the diffusion coefficient $G$ can be expressed in the following manner (see \cite[Proposition 2.4.5]{LR})
\begin{align}\label{250813_01}
\int_0^t G\left(u(s,\cdot)\right)\,dW(s) = \sum_{j=1}^{\infty} \sqrt{\lambda_j} \int_0^t g_j\left(u(s,\cdot)\right)\,d\mathcal{W}^j_s.
\end{align}
Moreover, as a consequence of Assumptions $(H_4)$ and $(H_6)$, one obtains the following lemma which will be used several times in the paper:
\begin{lem}\label{lemma_G1_G3}
For any $q \geq r$, there exists a constant $L_{G,q} > 0$ such that
\begin{align}\label{caro_11_02_a}
\left(\sum_{j=1}^\infty \|g_j'\|_{\infty}^q\right)^{\frac{1}{q}} \leq L_{G,q}.
\end{align}
In particular, for $q=2$, for all $v,w \in L^2(D)$,
\begin{align}\label{G3_inequality_G_q2}
\|G(v) - G(w)\|_{\HS}^2 \leq L_{G,2}^2 \|v - w\|_{L^2(D)}^2.
\end{align}
\end{lem}
\begin{proof}
If $q=r$, then \eqref{caro_11_02_a} holds directly by $(H_6)$. If $q>r$, from $(H_6)$ it follows that there exists $J\in \na^\star$ such that 
\begin{align*}
\sum_{j=1}^\infty \|g_j'\|_{\infty}^q
\leq &(J-1) L_{G,r}^q + L_{G,r}^r,
\end{align*}
and \eqref{caro_11_02_a} is obtained by setting $L_{G,q} := \left( (J-1) L_{G,r}^q + L_{G,r}^r \right)^{\frac{1}{q}}$. Then, \eqref{G3_inequality_G_q2} is a direct consequence of \eqref{caro_11_02_a} with $q=2$.
\end{proof}
At last, let us remark that whilst Assumption $(H_4)$ is common in the literature, $(H_5)$ and $(H_6)$ are not classic but seem not too restrictive for applications. Assumptions $(H_5)$ and $(H_6)$ will be needed in Proposition \ref{boundpe} because of the multiplicative interaction between the constraint $\psi(u)$ and the diffusion coefficient $G$ that is not necessarily Lipschitz continuous but an infinite sum of Lipschitz continuous functions.
\begin{ex}
Let us give some examples of functions $(g_j)_{j\in\na^\star}$ satisfying $(H_4)$-$(H_6)$. By setting $\alpha\in[1,2)$ and $\mu_j\in\ell^\alpha(\R)$ for any $j\in\na^\star$, we can consider the sequence $(g_j)_{j\in\na^\star}$ defined by $g_j: v \mapsto \mu_j \varphi(v)$, where $\varphi:\R\to\R$ is Lipschitz continuous with $\operatorname{supp}(\varphi)\subset [0,1]$. One can also take $g_j: v \mapsto \varphi(\mu_j v)\chi$, where $\varphi:\R \to \R$ is a bounded, Lipschitz continuous function and $\chi$ is a smooth cutoff function with $\operatorname{supp}(\chi)\subset [0,1]$, $0\leq \chi\leq 1$.
\end{ex}
\subsection{Motivation and state of the art}
In general, the proposed system of equations \eqref{equation}, of the phase-field type, models complex and evolving phenomena such as phase changes and damage. According to the model proposed in \cite{BBBLV} where we studied the particular case of Problem \eqref{equation0} with $p=2$ and a one-dimensional Brownian motion $(W(t))_{t\geq 0}$ from a theoretical point of view, we have in mind that $u$ is a parameter describing in continuum media the evolution of damage. Indeed, in the micro-structure of a material, $u$ locally represents the proportion of cohesive active bonds. Next, the multivalued term $\partial I_{[0,1]}(u)$ in \eqref{equation0} represents a physical constraint on $u$ which is forced to take values in the interval $[0,1]$. The $p$-Laplacian term $\plap u$ allows us to account for the microcracking zone (process zone), which has finite dimension even though it is difficult to measure. This non-local term can also be interpreted as an \enquote{homogenization} of the microcracked zone. The damage is diffuse, progressive, and widespread. Another argument is that without this term, the dissipated energy by damage is often mesh-dependent, which contradicts the fact that this energy is a material property. In the case of damage, the presence of a Laplacian term highlights the well-known non-local nature of the phenomenon \cite{fremond}. The more general $p$-Laplacian term can be used to enrich classical models \cite{LRR} and appears promising for modelling complex non-local effects. More specifically, the $p$-Laplacian allows for greater flexibility in modeling nonlinear phenomena, as it generalizes the classical Laplacian operator ($p=2$, treated in \cite{BBBLV}) to model diffusion processes \cite{Anguiano2023,BGKP18} where the velocity depends nonlinearly on the gradient (\textit{e.g.,} flow in porous media, heat diffusion in heterogeneous materials). Depending on the value of $p$, the model can capture slow diffusion regimes ($p>2$) or fast diffusion regimes ($1<p<2$), or even \enquote{sandpile} phenomena ($p\to\infty$) \cite{AMRT09}. The $p$-Laplacian is also used in image processing \cite{Kuijper2007} (denoising and segmentation), as it preserves edges better than the classical Laplacian, in fluids mechanics (modeling non-Newtonian flows such as threshold fluids), and in biology \cite{LIU2026104478} (describing cell migration or tumor growth). In addition, the stochastic nature of the system means that the variability of phenomena on a microscopic scale (crack opening and closure, micro-adhesion, etc.) can be taken into account. Stochasticity is introduced to the Allen-Cahn problem by a multiplicative stochastic force in the form of an It\^{o} integral.

As we mentioned in \cite{BSVZ2025}, stochastic Allen-Cahn equation as well as more general stochastic partial differential inclusions than the one proposed in this paper have been very popular research subjects in recent decades, both from a theoretical and numerical point of view (see the following contributions in chronological order \cite{Rascanu96,BarbuRascanu97,BensoussanRascanu97,RNT12,AKM16,TV20,B21,BOS22,DPGS22,DPGS23,SZ23,M23,SYVZ25,Tah25,DPFSZ26,MYRZ26} and \cite{FP03,QW19,BG20,ABNP21,LQ21,BGJK23,BP24,SZ25}). For a thorough exposition of these papers, we refer the reader to the introduction of \cite{BSVZ2025}.\\
To the best of our knowledge, an Allen-Cahn problem including a $p$-Laplace operator, particularly associated with Neumann boundary conditions, in a multivalued setting with a multiplicative stochastic force as presented in Problem\eqref{equation0} has not yet been considered in the literature. In the present study, we intend to fill this gap and to propose an existence and uniqueness result for such a problem. Our strategy consists in showing the existence of a solution through the convergence analysis of a semi-implicit time-discretization scheme for a regularized version of Problem\eqref{equation0}, as we did in a former study \cite{BBBLV} for a particular case of Problem\eqref{equation0} as mentioned above. The well-posedness of such a time-discretization scheme has been proved in the preliminary work \cite{BSSZ2026} by the Minty-Browder theorem.
Let us mention that the originality of our current approach, compared to \cite{BBBLV}, is that the passages to the limit with respect to the time step and the regularization parameter will be done simultaneously, following our recent study \cite{BSVZ2025} where we analysed the convergence of a time and space approximation scheme for the problem studied in \cite{BBBLV}. 
\subsection{Main result and outline of the paper}
For Problem \eqref{equation}, we are interested in the following solution type:
\begin{defn}\label{solution}
A pair of processes $(u,\psi)\in L^p_{\mathcal{P}_T}(\Omega_T;V)\times L^2_{\mathcal{P}_T}(\Omega_T;L^2(D))$ such that
\begin{align*}
u \in L^2(\Omega;\mathscr{C}([0,T];L^2(D)))
\end{align*}
is called a solution to Problem \eqref{equation} if a.e. in $\Omega\times (0,T)\times D$
\begin{equation*}
0\leq u \leq 1\text{ and }\psi\in \partial I_{[0,1]}(u),
\end{equation*}
and if \Pas \ and for all $t\in [0,T]$,
\begin{align}\label{130426_b}
\hspace*{-0.25cm}u(t)-u_0+\int_0^t \left(\plap u(s)+\psi(s)\right)\,ds =\int_0^t G(u(s))\,dW(s)+\int_0^t \left(\beta(u(s))+f(s)\right)\,ds
\end{align}
in $L^2(D)$.
\end{defn}
\begin{rem}\label{240527_01}
The process $u$ is, \textit{a priori}, predictable with values in $L^2(D)$. As a direct consequence of Pettis measurability theorem, we will show the predictability of $u$ with values in $V$ in Lemma \ref{addreg u}.
\end{rem}
The main result of this contribution is existence and uniqueness of a solution $(u,\psi)$ for Problem \eqref{equation} in the sense of Definition \ref{solution} as stated in the following theorem:
\begin{thm}\label{main}
Under assumptions $(H_1)$ to $(H_6)$, there exists a unique solution $(u,\psi)$ for Problem \eqref{equation} in the sense of Definition \ref{solution}.
\end{thm}
To present the proof of our main result, we structure this paper as follows.
In Section \ref{section2}, we will introduce a regularized version of our problem (with a parameter $\eps\in\R_+^\star$) for which we will construct a time discretization scheme (with a parameter $\tau>0$) and define a set of approximating sequences. In Section \ref{estimates}, we will derive stability estimates for these approximations, independently of $\eps$ and $\tau$. The boundedness results will allow us to pass to the limit with respect to $\eps$ and $\tau$ simultaneously. This will be done in Section \ref{Section4} by extracting weakly converging sequences. After identifying limits coming from the discretization of non-linear terms, we will be able to prove the existence of a solution for Problem \eqref{equation} in the sense of Definition \ref{solution}. Then, the rest of Section \ref{Section4} will be devoted to the proof of the uniqueness of such a solution.
\section{Time discretization scheme for a regularized version of our multivalued problem}\label{section2}
Following our previous approaches in \cite{BBBLV, BSVZ2025}, we use a standard regularization procedure for the maximal monotone operator $\partial I_{[0,1]}$ by considering a sequence of approximations \enquote{\`a la Moreau-Yosida} (see, \textit{e.g.,}\ \cite{barbu76,brezis73}) denoted by $(\pe)_{\eps\in\R_+^\star}$ and defined for all $\eps\in\R_+^\star$ and all $v\in\R$ by
 \begin{align}\label{SPPR}
 	\begin{aligned}
 		\pe(v)=-\frac{(v)^-}{\eps}+\frac{(v-1)^{+}}{\eps}=
 		\left\{\begin{array}{clll}
 			\di\frac{v}{\eps}&\text{if}&v\leq 0\smallskip\\
 			0&\text{if}&v\in[0,1]\smallskip\\
 			\di\frac{v-1}{\eps}&\text{if}&v\geq 1.
 		\end{array}\right.
 	\end{aligned}
 \end{align}
Then, for a fixed parameter $\eps\in\R_+^\star$, we are interested in constructing a time discretization scheme for the following parabolic equation:
\begin{align}\label{eqeps}
\begin{aligned}
du_\eps+\big(\plap u_\eps+\pe(u_\eps)\big)\,dt &=G(u_\eps)\,dW(t)+\left(\beta(u_\eps)+f\right)\,dt, &&\text{ in }\Omega\times(0,T) \times D;\\
u_\eps(0,\cdot)&=u_0, &&\text{ in } \Omega\times D;\\
\nabla u_\eps\cdot \mathbf{n}&=0, &&\text{ on }\Omega\times(0,T)\times\partial D.
\end{aligned}
\end{align}
For $N\in\na^\star$, let $0 = t_0 \leq t_1 \leq \ldots \leq t_N = T$ be a subdivision of the interval $[0,T]$ with a time step $\tau=T/N=t_{n+1}-t_{n}$ for all $n\in\{0,\ldots,N-1\}$. For a fixed $\eps\in\R_+^\star$, we set $u_\eps^0=u_0$. We define our Euler-Maruyama semi-implicit scheme as follows.
Given a $\mathcal{F}_{t_n}$-measurable random variable $\upsn$ with values in $L^2(D)$ for a fixed $n\in \{0,\ldots,N-1\}$, we look for a $\mathcal{F}_{t_{n+1}}$-measurable random variable labelled by $\upsnp$, with values in $V$, such that the following identity holds \Pas , in $L^2(D)$
\begin{align}\label{240911_01}
&\upsnp-\upsn+\tau\left(\plap \upsnp+\pe(\upsnp)\right)=G(\upsn)\Delta_{n+1} W +\tau\left(\beta(\upsnp)+f_n\right),
\end{align}
where $\Delta_{n+1}W = W(t_{n+1})-W(t_{n})$ and $\displaystyle f_n = \frac{1}{\tau} \int_{t_n}^{t_{n+1}} f(s)\,ds$, for all $n\in\{0,\ldots,N-1\}$.
\begin{rem}
Let us mention that owing to the $\mathcal{F}_{t_n}$-measurability of $\upsn$ with values in $L^2(D)$ we have (\cite[Section 4.2.1]{DPZ14})
\begin{align*}
G(\upsn)\Delta_{n+1} W
= \int_{t_{n}}^{t_{n+1}}G(\upsn) \ dW(t)
&= \sum_{j=1}^{\infty}\big(G(\upsn)\circ \mathcal{Q}^{1/2}\big)(e_j)\left(\mathcal{W}^j(t_{n+1})-\mathcal{W}^j(t_n)\right)\nonumber,
\end{align*}
where the infinite sum is an element of $L^2\left(\Omega;L^2(D)\right)$ because $\Delta_{n+1}W$ values are in the Hilbert space $U$ and $G(\upsn)\circ \mathcal{Q}^{1/2}: U \to L^2(D)$ is a Hilbert-Schmidt operator. Moreover, due to assumption $(H_4)$, we arrive at
\begin{align*}
&G(\upsn)\Delta_{n+1} W = \sum_{j=1}^{\infty}g_j(\upsn)\left(\mathcal{W}^j(t_{n+1})-\mathcal{W}^j(t_n)\right).
\end{align*}
\end{rem}
By the Minty-Browder theorem, we proved in \cite{BSSZ2026} the well-posedness of the scheme as stated in the following proposition.
\begin{prop}\label{well-posedness_scheme}
Let us assume that Hypotheses $(H_1)$ to $(H_4)$ and $(H_6)$ are satisfied. Consider a fixed parameter $\eps\in\R_+^\star$, a fixed $N\in\na^\star$, define $\tau=T/N$ and $t_n=n\tau$ for any $ n\in\{0,\ldots,N\}$, and set $u_\eps^0=u_0$. Let $n\in\{0,\ldots,N-1\}$. For any $\mathcal{F}_{t_n}$-measurable random variable $\upsn$ taking values in $L^2(D)$, if $\tau<\frac{1}{L_\beta}$ then there exists a unique $\mathcal{F}_{t_{n+1}}$-measurable random variable $\upsnp$ with values in $V$ solving \eqref{240911_01}.
\end{prop}
Using the sequence of discrete approximations $(\upsnp)_{0\le n \le N-1}$ introduced by Proposition \ref{well-posedness_scheme}, we are able to construct the following temporal approximations of different types (right-continuous and piecewise-constant in time or right-continuous and piecewise affine in time, or simply time-continuous):
\begin{defn}\label{piecewiseconstantbis}
Let $\eps\in\R_+^\star$, $N\in\na^\star$ and set $\tau=T/N$ and $t_n=n\tau$, for any $n$ in $\{0,\ldots,N\}$. We define on $\Omega\times [0,T]$ the following approximations processes, for any $(\omega, t)\in \Omega\times [0,T]$:
\begin{itemize}
\item[$\bullet$] $\uhnr(\omega,t):= \displaystyle \sum_{n=0}^{N-1}\upsnp(\omega)\,\mathds{1}_{[t_n,t_{n+1})}(t) \ \text{ if } t\in [0,T)$ and $\uhnr(\omega,T):= \ups^N(\omega)$.
\item[$\bullet$] $\uhnl(\omega,t):=\displaystyle \sum_{n=0}^{N-1} \upsn(\omega)\,\mathds{1}_{[t_n,t_{n+1})}(t) \ \text{ if } t\in [0,T)$ and $\uhnl(\omega,T) := \ups^N(\omega)$.
\item[$\bullet$] $\widehat{u}_{N,\eps}(\omega,t) :=\displaystyle \sum_{n=0}^{N-1}\left( \frac{\upsnp(\omega)-\upsn(\omega)}{\tau}(t-t_n)+\upsn(\omega)\right)\mathds{1}_{[t_n,t_{n+1})}(t) \ \text{ if } t\in [0,T)$ and
$\widehat{u}_{N,\eps}(\omega, T) := \ups^N(\omega)$.
\item[$\bullet$] $M_{N,\eps}(\omega,t) :\displaystyle =\int_0^t G(\uhnl(\omega,s))\,dW(s)$.
\item[$\bullet$] $\widehat{M}_{N,\eps}(\omega,t) :=\displaystyle \sum_{n=0}^{N-1} \left(\frac{M^{n+1}_\eps(\omega)-M^n_\eps(\omega)}{\tau}(t-t_n)+M^n_\eps(\omega)\right)\mathds{1}_{[t_n,t_{n+1})}(t) \ \text{ if } t\in [0,T)$ and $\widehat{M}_{N,\eps}(\omega, T) := M^N_\eps(\omega),$ where for any $n\in \{0,\ldots,N\}$, $M^n_\eps(\omega)=M_{N,\eps}(\omega,t_n)$.
\item[$\bullet$] $f^N(\omega, t):=\displaystyle \sum_{n=0}^{N-1} \frac{1}{\tau} \int_{t_n}^{t_{n+1}} f(\omega,s)\,ds\,\mathds{1}_{[t_n,t_{n+1})}(t) \ \text{ if } t\in [0,T)$ and $f^N(\omega,T):=f(\omega,T)$.
\end{itemize}
\end{defn}
In order to relieve the presentation of the paper, we will omit most of the time the random variable $\omega$.
\section{Stability estimates}\label{estimates}
In this section, we will firstly obtain uniform bounds with respect to $\eps$ and $N$ for the sequence of discrete approximations $(\upsnp)_{0\le n \le N-1}$ given by Proposition \ref{well-posedness_scheme}. Secondly, these bounds will allow us to derive stability estimates
satisfied by the following sequences of discrete approximations (up to some technical conditions on the indexes $N$ and $\eps$)constructed according to Definition \ref{piecewiseconstantbis}:
\begin{equation}\label{040526_a}
(v_{N,\eps})_{N,\eps}, \ \left(\rho_{\eps}(v_{N,\eps})\right)_{N,\eps}, \ \left(M_{N,\eps}-\widehat{M}_{N,\eps}\right)_{N,\eps} \text{ and } \left(\partial_t(\widehat{u}_{N,\eps}-\widehat M_{N,\eps})\right)_{N,\eps},
\end{equation}
where for any
$(N,\eps)\in\na^\star\times\R_+^\star$, with $N$ sufficiently large,
$v_{N,\eps}\in\{\uhnr,\uhnl\}$ and $\rho_{\eps}\in\{\pe,\beta,G,-\frac{(\cdot)^-}{\eps},\frac{(\cdot-1)^+}{\eps}\}$.
Before stating and proving all these boundedness results, let us begin with a convergence property satisfied by the time-approximation of the additive source term $f$:
\begin{lem}\label{CVfhnl}
The sequence of piecewise constant processes $(f^N)_N$ converges to $f$ strongly in $L^2(\Omega;L^2(0,T;L^2(D)))$ as $N\to\infty$.
\end{lem}
\begin{proof}
We obtain the result by combining a standard argument of Steklov average (see, \textit{e.g.,} \cite[Lemma 2.1.2]{Wu2001}) with Lebesgue's dominated convergence theorem for Bochner spaces (see, \textit{e.g.,} \cite[Theorem 1.3.3]{D}).
\end{proof}
Let us now address uniform bounds with respect to $\eps$ and $N$ for the sequence of discrete approximations $(\upsnp)_{0\le n \le N-1}$ given by Proposition \ref{well-posedness_scheme}:
\begin{prop}\label{bounds}
There exists a constant $K_0\in\R_+^\star$ that depends only on $L_{G,2}$, $L_\beta$, $f$, $u_0$ and $T$, such that for any $\eps\in\R_+^\star$ and any $N\in\na^\star$, if $\frac{T}{N} \leq \frac{3}{4} \frac{1}{1 + 2 L_\beta}$, then
\begin{align*}
&\erwb \|\upsn\|_{L^2(D)}^2 \erwe+\sum_{k=0}^{n-1}\erwb\|\upskp-\upsk\|_{L^2(D)}^2\erwe+\tau\sum_{k=0}^{n-1}\erww{\| \upskp\|_V^p}\leq K_0,\; \forall n\in \{1,\ldots,N\},
\end{align*}
where $\erw$ denotes the expectation \textit{i.e.} the integral with respect to the probability measure $\mathds{P}$, over $\Omega$.
\end{prop}
\begin{proof}
Set $\eps\in\R_+^\star$, $N\in\na^\star$, assume that $\tau=\frac{T}{N}\leq \frac{3}{4} \frac{1}{1 + 2 L_\beta}$ and fix $n\in \{1,\ldots,N\}$. For any $k\in \{0,\ldots,n-1\}$, multiply the scheme \eqref{240911_01} by $\upskp$, take the integral over $D$ and the expectation to obtain
\begin{align}\label{implicitscheme}
&\erww{(\upskp-\upsk,\upskp)_{L^2(D)}}
+\tau \erww{(\plap \upskp, \upskp)_{L^2(D)}}+\tau \erwb \left(\pe(\upskp),\upskp\right)_{L^2(D)}\erwe\nonumber\\
=\,&\erww{(G(\upsk)\Delta_{k+1}W, \upskp)_{L^2(D)}}+\tau \erww{\left(\beta(\upskp)+f_k,\upskp\right)_{L^2(D)}},
\end{align}
where $(\cdot,\cdot)_{L^2(D)}$ denotes the canonical inner product in $L^2(D)$.
We consider the terms of \eqref{implicitscheme} separately. Firstly note that
\begin{align}\label{term1}
\erwb(\upskp-\upsk,\upskp)_{L^2(D)}\erwe=\frac{1}{2}\erwb\|\upskp\|_{L^2(D)}^2-\|\upsk\|_{L^2(D)}^2+\|\upskp-\upsk\|_{L^2(D)}^2\erwe.
\end{align}
Secondly, due to the continuous injection of $L^2(D)$ into $V^*$, we will use repeatedly throughout the paper the identification $(u,v)_{L^2(D)}=\left\langle u, v \right\rangle_{V^*,V}$, for any $u\in L^2(D)$, $v\in V$, which allows us to write that
\begin{align}\label{term2}
&(\plap \upskp, \upskp)_{L^2(D)}
=\langle\plap \upskp, \upskp\rangle_{V^*,V} \nonumber\\
&= \int_D \left( |\nabla \upskp(x)|^{p-2} \nabla \upskp(x) \cdot \nabla \upskp(x)
+ |\upskp(x)|^{p-2} \upskp(x) \upskp(x) \right)\,dx\nonumber\\
&= \int_D \left(|\nabla \upskp(x)|^{p} + |\upskp(x)|^p \right)\,dx
= \|\upskp\|_V^p.
\end{align}
Thirdly, as $\pe(0)=0$ and $\pe$ is monotone,
\begin{align}\label{term3}
\tau \erww{(\pe(\upskp),\upskp)_{L^2(D)}}\geq 0.
\end{align}
Fourthly, since $G(\upsk)$ is $\mathcal{F}_{t_k}$-measurable and $\Delta_{k+1}W$ is $\mathcal{F}_{t_k}$-independent, one obtains
\begin{align*}
\erww{(G(\upsk)\Delta_{k+1}W, \upsk)_{L^2(D)}}=0
\end{align*}
and so, by Young's inequality, one arrives at
\begin{align*}
\erww{(G(\upsk)\Delta_{k+1}W, \upskp)_{L^2(D)}}&=\erww{(G(\upsk)\Delta_{k+1}W, \upskp-\upsk)_{L^2(D)}}\\
&\leq\erww{\|G(\upsk)\Delta_{k+1}W\|_{L^2(D)}^2}+\frac{1}{4}\erww{\|\upskp-\upsk\|_{L^2(D)}^2}.
\end{align*}
Using \cite[Proposition 4.20 and Remark 4.21 p. 97]{DPZ14} and Inequality \eqref{G3_inequality_G_q2}, one has
\begin{align*}
\erww{\|G(\upsk)\Delta_{k+1}W\|_{L^2(D)}^2}
&=\erww{\left\| \int_{t_k}^{t_{k+1}}G(\upsk)\,dW(t)\right\|_{L^2(D)}^2}\\
&=\erww{\int_{t_k}^{t_{k+1}}\| G(\upsk)\|_{\HS}^2\,dt}\\
&=\tau\erww{\| G(\upsk)\|_{\HS}^2}\\
&\leq \tau L_{G,2}^2 \erww{\| \upsk\|_{L^2(D)}^2},
\end{align*}
and then
\begin{align}\label{term4}
\erww{(G(\upsk)\Delta_{k+1}W, \upskp)_{L^2(D)}}&\leq \tau L_{G,2}^2 \erww{\| \upsk\|_{L^2(D)}^2}+\frac{1}{4}\erww{\|\upskp-\upsk\|_{L^2(D)}^2}.
\end{align}
Fifthly, using the Lipschitz continuity of $\beta$ with $\beta(0)=0$, the following holds
\begin{align}\label{term5}
\erww{(\beta(\upskp),\upskp)_{L^2(D)}}\leq L_\beta\erww{\|\upskp\|_{L^2(D)}^2}.
\end{align}
Finally, by Young's inequality,
\begin{align}\label{term6}
\erww{(f_k,\upskp)_{L^2(D)}}
\leq\,& \frac{1}{2}\erww{\| f_k\|_{L^2(D)}^2}
+\frac{1}{2} \erww{\| \upskp\|_{L^2(D)}^2}.
\end{align}
\noindent Combining \eqref{term1},\eqref{term2},\eqref{term3},\eqref{term4}, \eqref{term5} and \eqref{term6}, one gets
\begin{align}\label{before_discarding}
&\frac{1}{2} \erwb\|\upskp\|_{L^2(D)}^2-\|\upsk\|_{L^2(D)}^2+\|\upskp-\upsk\|_{L^2(D)}^2\erwe
+\tau \erww{\|\upskp\|_V^p}\nonumber\\
\leq\,&
\tau L_{G,2}^2 \erww{\| \upsk\|_{L^2(D)}^2}
+\frac{1}{4} \erww{\|\upskp-\upsk\|_{L^2(D)}^2}
+\tau L_\beta \erww{\|\upskp\|_{L^2(D)}^2}\nonumber\\
&+\frac{\tau}{2} \erww{\| f_k\|_{L^2(D)}^2}
+\frac{\tau}{2} \erww{\| \upskp\|_{L^2(D)}^2}.
\end{align}
We discard the nonnegative term $\tau \erww{\|\upskp\|_V^p}$, multiply the above equation by $2$, set $\delta = 1-\tau(2 L_\beta+1)$ and arrive at
\begin{align*}
&\delta \erwb\|\upskp\|_{L^2(D)}^2\erwe
-\erwb\|\upsk\|_{L^2(D)}^2\erwe
+ \frac{1}{2} \erwb\|\upskp-\upsk\|_{L^2(D)}^2\erwe\\
\leq\ & 2\tau L_{G,2}^2 \erww{\| \upsk\|_{L^2(D)}^2}+\tau\erww{\| f_k\|_{L^2(D)}^2}.
\end{align*}
Adding $\tau(2 L_\beta+1)\erwb\|\upsk\|_{L^2(D)}^2\erwe$ on both sides of the inequality, and discarding the nonnegative term $\frac{1}{2} \erwb\|\upskp-\upsk\|_{L^2(D)}^2\erwe$ we get
\begin{align*}
&\delta \erwb\|\upskp\|_{L^2(D)}^2-\|\upsk\|_{L^2(D)}^2\erwe
\leq\tau (2 L_{G,2}^2 + 2 L_\beta + 1) \erww{\|\upsk\|_{L^2(D)}^2}
+\tau\erww{\| f_k\|_{L^2(D)}^2}.
\end{align*}
After summing over $k\in\{0,\ldots,n-1\}$ and recalling that $\upsz=u_0$ one arrives at
\begin{align*}
\delta \erwb\|\upsn\|_{L^2(D)}^2\erwe\leq 
\delta \erwb\|u_0\|_{L^2(D)}^2\erwe
+\tau (2 L_{G,2}^2 + 2 L_\beta + 1) \sum_{k=0}^{n-1} \erww{\|\upsk\|_{L^2(D)}^2}
+\|f\|_{L^2(\Omega_T;L^2(D))}^2.
\end{align*}
By assumption, $\tau \leq \frac{3}{4} \frac{1}{1 + 2 L_\beta}$ so $\frac{1}{4} \leq \delta \leq 1$, and using this we arrive at the following inequality
\begin{align*}
\erwb\|\upsn\|_{L^2(D)}^2\erwe
\leq&\,
4 \left(
\erwb\|u_0\|_{L^2(D)}^2\erwe
+\| f \|_{L^2(\Omega_T;L^2(D))}^2
\right)
\\
&
+4 \tau (2 L_{G,2}^2 + 2 L_\beta + 1) \sum_{k=0}^{n-1} \erww{\| \upsk\|_{L^2(D)}^2}.
\end{align*}
Applying the discrete Grönwall lemma (see special cases in \cite[Corollary 1]{Beesack1985}) yields for all $n \in \{1,\ldots,N\}$
\begin{align}\label{uhnbound}
&\erwb\|\upsn\|_{L^2(D)}^2\erwe
\leq\,
4 \left(
\erwb\|\upsz\|_{L^2(D)}^2\erwe
+\| f \|_{L^2(\Omega_T;L^2(D))}^2
\right)
e^{4 T (2 L_{G,2}^2 + 2 L_\beta + 1)}.
\end{align}
Since the initial condition $\upsz = u_0$ is in $L^2(\Omega;L^2(D))$ by $(H_1)$, and using \eqref{uhnbound}, there exists a constant $\Upsilon>0$ such satisfying
\begin{align}\label{uhnboundbis}
\sup_{n\in\{0,\ldots,N\}} \erwb\|\upsn\|_{L^2(D)}^2\erwe \leq \Upsilon.
\end{align}
Going back to \eqref{before_discarding}, summing over $k\in \{0,\ldots,n-1\}$ and using \eqref{uhnboundbis}, we arrive at
\begin{align*}
&\erwb\|\upsn\|_{L^2(D)}^2\erwe
+\frac{1}{2} \sum_{k=0}^{n-1}\erwb\|\upskp-\upsk\|_{L^2(D)}^2\erwe
+2 \tau \sum_{k=0}^{n-1} \erww{\|\upskp\|_V^p}\\
\leq\,&
 \| u_0 \|_{L^2(\Omega;L^2(D))}^2
+ T \Upsilon (2 L_{G,2}^2 + L_\beta + 1)
+ \| f \|_{L^2(\Omega_T;L^2(D))}^2,
\end{align*}
and the announced result holds.
\end{proof}
Thanks to Proposition \ref{bounds}, we are now able to derive stability estimates satisfied by the sequences listed in \eqref{040526_a}.
\begin{lem}\label{210611_lem01}
For $N$ sufficiently large, we have the following boundedness results independently of the discretization and regularization parameters $N\in\na^\star$ and $\eps\in\R_+^\star$:
\begin{itemize}
\item The sequences $(\uhnr)_{N,\eps}$ and $(\uhnl)_{N,\eps}$ are bounded in $L^\infty(0,T;L^2(\Omega;L^2(D)))$.
\item The sequence $(\uhnr)_{N,\eps}$ is bounded in $L^p(\Omega;L^p(0,T;V))$.
\item The sequence $(\uhnl)_{N,\eps}$ is bounded in $L^2_{\mathcal{P}_T}\big(\Omega_T;L^2(D)\big)$.
\end{itemize}
\end{lem}
\begin{proof}
We note that by \eqref{uhnboundbis}, there exists a constant $\Upsilon>0$ such that for any $N\in \na^\star$ sufficiently large and any $\eps\in\R_+^\star$
\begin{align*}
\| \uhnr \|_{L^{\infty}(0,T;L^2(\Omega;L^2(D)))} + \| \uhnl \|_{L^{\infty}(0,T;L^2(\Omega;L^2(D)))}
& \leq 2 \sup_{n\in \{0,\ldots,N\}} \erwb\|\upsn\|_{L^2(D)}^2\erwe
\leq 2 \Upsilon.
\end{align*}
The sequence $(\uhnr)_{N,\eps}$ is bounded in $L^p(\Omega;L^p(0,T;V))$ by $K_0^{\frac{1}{p}}$, where the constant $K_0$ is given by Proposition \ref{bounds}.
As a consequence of the $\mathcal{F}_{t_n}$-measurability of the random variable $u^n_{\eps}$ for all $n\in\{0,\ldots,N\}$, the process $(\uhnl)_{N,\eps}$ is predictable with values in $L^2(D)$. This is due to the construction of $(\uhnl)_{N,\eps}$ as an elementary process adapted to the filtration $(\mathcal{F}_t)_{t\geq 0}$, making it predictable.
\end{proof}
\begin{rem}\label{251002_r1}
By Proposition~\ref{bounds}, one gets the following useful estimates for any $(N,\eps)\in \na^\star\times \R_+^{\star}$ with $N$ sufficiently large:
\begin{align}
\|\uhnr-\uhnl\|_{L^2(\Omega;L^2(0,T;L^2(D)))}^2=\tau\erww{\sum_{n=0}^{N-1}\|u_\eps^{n+1}-u_\eps^n\|_{L^2(D)}^2}
\leq\ & K_0\tau,\label{210824_05}\\
\text{and }\|\pe(\uhnr)-\pe(\uhnl)\|_{L^2(\Omega;L^2(0,T;L^2(D)))}^2 \leq\ & \frac{\tau}{\eps^2}K_0. \label{controldiffpe(uhnl)pe(uhnr)}
\end{align}
\end{rem}
\begin{lem}\label{boundguhnlr_and_betauhnlr}
The sequences $(G(\uhnr))_{N,\eps}$ and $(G(\uhnl))_{N,\eps}$ are bounded in \linebreak $L^2(\Omega;L^2(0,T;\HS)$, the sequences $(\beta(\uhnr))_{N,\eps}$, and $(\beta(\uhnl))_{N,\eps}$ are bounded in $L^2(\Omega;L^2(0,T;L^2(D)))$, independently of the discretization and regularization parameters $N\in\na^\star$ and $\eps\in\R_+^\star$, for $N$ sufficiently large. Moreover, for any $(N,\eps)\in\na^\star\times\R_+^\star$, $G(\uhnl)$ and $\beta(\uhnl)$ are predictable processes with values in $\HS$ and $L^2(D)$, respectively.
\end{lem} 
\begin{proof} 
Thanks to \eqref{uhnboundbis}, one gets that for all $n\in\{1,\ldots N\}$
\begin{align*}
\tau \sum_{k=0}^{n-1} \erww{\|G(\upsk)\|_{\HS}^2}
\leq L_{G,2}^2 \tau \sum_{k=0}^{n-1} \erwb\|\upsk\|_{L^2(D)}^2\erwe
\leq L_{G,2}^2 T \Upsilon.
\end{align*}
The boundedness of $(G(\uhnr))_{N,\eps}$ and $(G(\uhnl))_{N,\eps}$ in $L^2(\Omega;L^2(0,T;\HS)$ results directly. The boundedness in $L^2(\Omega;L^2(0,T;L^2(D)))$ of $(\beta(\uhnr))_{N,\eps}$, and $(\beta(\uhnl))_{N,\eps}$ is a consequence of the boundedness of the sequences $(\uhnr)_{N,\eps}$ and $(\uhnl)_{N,\eps}$ in the same space (given by Proposition \ref{bounds}), combined with the Lipschitz continuity of $\beta$. At last, for any fixed $(N,\eps) \in \na^\star \times \R_+^\star$, the predictability of $G(\uhnl)$ and $\beta(\uhnl)$ results from the predictability of $\uhnl$.
\end{proof}
\begin{prop}\label{boundpe}
Given the constant $r\in[1,2)$ of $(H_6)$, if we assume that there exists $\theta\in(0,\frac{4}{r}-2]$ such that $\frac{T}{N}=\mathcal{O}(\eps^{2+\theta})$, then there exists a constant $K_1 \geq 0$ not depending on the discretization and regularization parameters $N\in\na^\star$ and $\eps\in\R_+^\star$ such that, for $N$ sufficiently large
\begin{equation*}
\sum_{k=0}^{N-1}\tau\erww{\|\pe(\upskp)\|_{L^2(D)}^2}\leq K_1.
\end{equation*}
\end{prop}
\begin{proof}
Assuming $(H_6)$, setting $\theta \in (0,\frac{4}{r}-2]$, $N\in \na^\star$ and $\eps\in \mathbb R_+^\star$, multiplying \eqref{240911_01} by $\pe(\upskp)$, integrating over $D$, taking the expectation and summing over $k\in \{0, \ldots, N-1\}$, we obtain
\begin{align}\label{scmpeunpk_ced}
\begin{aligned}
&\sum_{k=0}^{N-1} \erwb(\upskp-\upsk,\pe(\upskp))_{L^2(D)}\erwe
+\tau \sum_{k=0}^{N-1} \erww{\left(\plap\upskp,\pe(\upskp)\right)_{L^2(D)}}\\
&+\sum_{k=0}^{N-1} \tau \erww{\|\pe(\upskp)\|_{L^2(D)}^2}
=\,\sum_{k=0}^{N-1} \erww{(G(\upsk)\Delta_{k+1}W, \pe(\upskp))_{L^2(D)}}\\
&+\sum_{k=0}^{N-1} \tau \erww{(\beta(\upskp)+f_k,\pe(\upskp))_{L^2(D)}}.
\end{aligned}
\end{align}
$\bullet$ For the study of the first term on the left-hand side of \eqref{scmpeunpk_ced}, we introduce, following \cite{BSVZ2025}, the convex antiderivative $\phe$ of $\pe$ such that for any $v\in \R$
\begin{align*}\label{defphe_ced}
\phe(v)=
\left\{\begin{array}{clll}
\di\frac{v^2}{2\eps}&\text{if}&v\leq 0\\
0&\text{if}&v\in[0,1]\\
\di\frac{(v-1)^2}{2\eps}&\text{if}&v\geq 1.
\end{array}\right.
\end{align*}
By convexity of $\phe$, the inequality
\begin{equation*}
(\upskp-\upsk)\pe(\upskp)=(\upskp-\upsk)\phe'(\upskp)\geq \phe(\upskp)-\phe(\upsk)
\end{equation*}
holds a.e. in $\Omega\times D$
and so
\begin{align}\label{T1_ced}
\sum_{k=0}^{N-1}\erww{\left(\upskp-\upsk,\pe(\upskp)\right)_{L^2(D)}}\geq 0,
\end{align}
owing to the facts that $\erwb \phe(u^{N}_\eps)\erwe \geq 0$ and $\erwb \phe(u_{0})\erwe=0$ since, from Assumption $(H_1)$, $0\leq u_0 \leq 1$ a.e. in $\Omega\times D$.\\
$\bullet$ Noting that $\nabla \pe(\upskp)=\pe'(\upskp)\nabla\upskp$ (see \cite{MM79}), recalling that
\begin{align*}
\left(\plap\upskp,\pe(\upskp)\right)_{L^2(D)}
&=\langle\plap\upskp,\pe(\upskp)\rangle_{V^*,V}\\
&=\int_D \left(|\nabla\upskp(x)|^p\pe'(\upskp(x))+|\upskp(x)|^{p-2}\upskp(x)\pe(\upskp(x))\right)\,dx
\end{align*}
and using the monotonicity of $\pe$, we arrive at
\begin{align}\label{T2_ced}
\tau \sum_{k=0}^{N-1} \erww{\left(\plap\upskp,\pe(\upskp)\right)_{L^2(D)}}
\geq 0.
\end{align}
$\bullet$ Using the definition of the stochastic integral\eqref{250813_01} and the fact that $g_j(\upsk)\pe(\upsk)=0$ for all $j\in\na^\star$ and a.e. in $\Omega\times D$ we get
\begin{align*}
&G(\upsk)\Delta_{k+1}W\pe(\upskp)=\int_{t_k}^{t_{k+1}}G(\upsk)\,dW(s)\pe(\upskp)
\\&=\sum_{j=1}^{\infty}\sqrt{\lambda_j}\int_{t_k}^{t_{k+1}}g_j(\upsk)\,d\mathcal{W}^j_s\pe(\upskp)\\
&=\sum_{j=1}^{\infty}\sqrt{\lambda_j}\Delta_{k+1}\mathcal{W}^j g_j(\upsk) \left(\pe(\upskp)-\pe(\upsk)\right)
\end{align*}
where $\Delta_{k+1}\mathcal{W}^j:=\mathcal{W}^j_{t_{k+1}}-\mathcal{W}^j_{t_k}$ for all $j\in\na^\star$. Moreover, according to \cite{BSVZ2025}, for all $j\in\na^\star$ the following inequality holds true a.e. in $\Omega\times D$
\begin{align*}
\left|g_j(\upsk)\left(\pe(\upskp)-\pe(\upsk)\right)\right| \leq \frac1\eps \left|g_j(\upskp)-g_j(\upsk)\right\|\upskp-\upsk|.
\end{align*}
Consequently, we get
\begin{align}\label{250813_02_ced}
\begin{aligned}
&G(\upsk)\Delta_{k+1}W\pe(\upskp)
\leq\sum_{j=1}^{\infty}\sqrt{\lambda_j}|\Delta_{k+1}\mathcal{W}^j|\frac1\eps \left|g_j(\upsk)-g_j(\upskp)\right\|\upskp-\upsk|.
\end{aligned}
\end{align}
By applying Young's and Cauchy-Schwarz inequalities on the right-hand side of \eqref{250813_02_ced}, we arrive at
\begin{align}\label{250813_03_ced}
&G(\upsk)\Delta_{k+1}W\pe(\upskp)\nonumber\\
\leq\ & \frac{1}{2}|\upskp-\upsk|^2+\frac{1}{2}\left(\sum_{j=1}^{\infty}\sqrt{\lambda_j}|\Delta_{k+1}\mathcal{W}^j|\frac1\eps \left|g_j(\upsk)-g_j(\upskp)\right|\right)^2\\
\leq\ & \frac{1}{2}|\upskp-\upsk|^2+\frac{1}{2}\left(\sum_{j=1}^{\infty}\lambda_j\frac{1}{\eps^2}|\Delta_{k+1}\mathcal{W}^j|^2\right)\left(\sum_{j=1}^{\infty}\left|g_j(\upsk)-g_j(\upskp)\right|^2\right).\nonumber
\end{align}
Now, the second member in the sum on the right-hand side of \eqref{250813_03_ced} deserves our attention. Using Young and H\"older' inequalities with parameter $q=1+\frac{2}{\theta}$ and its conjugate $\frac{q}{q-1}=1+\frac{\theta}{2}$, writing $\lambda_j=\lambda_j^{1/q}\lambda_j^{(q-1)/q}$, and recalling that $\sum_{j=1}^{\infty}\lambda_j=\operatorname{Tr}(\mathcal{Q})$ we have
\begin{align}\label{250813_04_ced}
&\left(\sum_{j=1}^{\infty}\lambda_j\frac{1}{\eps^2}|\Delta_{k+1}\mathcal{W}^j|^2\right)\left(\sum_{j=1}^{\infty}\left|g_j(\upsk)-g_j(\upskp)\right|^2\right)\nonumber\\
&\leq \frac{1}{q}\left(\sum_{j=1}^{\infty}\lambda_j\frac{1}{\eps^2}|\Delta_{k+1}\mathcal{W}^j|^2\right)^q+\frac{q-1}{q}\left(\sum_{j=1}^{\infty}\left|g_j(\upsk)-g_j(\upskp)\right|^2\right)^{\frac{q}{q-1}}\\
&\leq \frac{1}{q}\left(\operatorname{Tr}(\mathcal{Q})\right)^{q-1}\left(\sum_{j=1}^{\infty}\lambda_j\frac{1}{\eps^{2q}}|\Delta_{k+1}\mathcal{W}^j|^{2q}\right)+\frac{q-1}{q}\left(\sum_{j=1}^{\infty}\left|g_j(\upsk)-g_j(\upskp)\right|^2\right)^{\frac{q}{q-1}}.\nonumber
\end{align}
By using the Lipschitz property stated in $(H_4)$, the constant $K_g$ given by $(H_5)$ and after noticing that $\frac{2}{q}+\frac{2(q-1)}{q}=2$, it follows that
\begin{align*}
\begin{aligned}
&\left(\sum_{j=1}^{\infty}\left|g_j(\upsk)-g_j(\upskp)\right|^2\right)^{\frac{q}{q-1}}\\
=\ & \left(\sum_{j=1}^{\infty}\left|g_j(\upsk)-g_j(\upskp)\right|^{\frac{2}{q}} \left|g_j(\upsk)-g_j(\upskp)\right|^{\frac{2(q-1)}{q}} \right)^{\frac{q}{q-1}}\\
 \leq\ & \left(2\sup_{j\in\na^\star}\| g_j\|_{\infty}\right)^{\frac{2}{q-1}}\left(\sum_{j=1}^{\infty}\left|g_j(\upsk)-g_j(\upskp)\right|^{\frac{2(q-1)}{q}}\right)^{\frac{q}{q-1}}\\
\leq\ & (2K_g)^{\frac{2}{q-1}} \left(\sum_{j=1}^{\infty}\left|g_j(\upsk)-g_j(\upskp)\right|^{\frac{2(q-1)}{q}}\right)^{\frac{q}{q-1}}\\
\leq\ & (2K_g)^{\frac{2}{q-1}}\left(\sum_{j=1}^{\infty} \|g_j'\|_\infty^{\frac{2(q-1)}{q}}\right)^{\frac{q}{q-1}} \left(|\upsk-\upskp|^{\frac{2(q-1)}{q}}\right)^{\frac{q}{q-1}}
.
\end{aligned}
\end{align*}
Now, owing to the fact that $\frac{2(q-1)}{q}=\frac{4}{2+\theta}$ and since $\theta \in (0,\frac{4}{r}-2]$ by assumption, one gets that $\frac{2(q-1)}{q} \geq r$, which allows us to apply Lemma \ref{lemma_G1_G3} and we get the existence of a constant $L_{G,\frac{2(q-1)}{q}} > 0$ such that
\begin{equation*}
\left(\sum_{j=1}^{\infty} \|g_j'\|_\infty^{\frac{2(q-1)}{q}}\right)^{\frac{q}{2(q-1)}}
\leq L_{G,\frac{2(q-1)}{q}},
\end{equation*}
and then by setting $C_1:=(2K_g)^{\frac{2}{q-1}} \left(L_{G,\frac{2(q-1)}{q}}\right)^2 \geq 0$, one gets
\begin{align}\label{250813_05_ced}
&\left(\sum_{j=1}^{\infty}\left|g_j(\upsk)-g_j(\upskp)\right|^2\right)^{\frac{q}{q-1}}
\leq C_1 |\upskp-\upsk|^2.
\end{align}
Denoting by $|D|$ the $d$-dimensional Lebesgue measure of $D$, setting $C_2:=\frac{1}{q}(\operatorname{Tr}(\mathcal{Q}))^{q-1}|D|$ and plugging \eqref{250813_04_ced} and \eqref{250813_05_ced} into \eqref{250813_03_ced} we get
\begin{align*}
&\erww{\left(G(\upsk)\Delta_{k+1}W, \pe(\upskp)\right)_{L^2(D)}}\\
\leq& \frac{1}{2}\erww{\left(1+\frac{C_1(q-1)}{q}\right)\|\upskp-\upsk\|_{L^2(D)}^2
+C_2 \sum_{j=1}^{\infty}\lambda_j\frac{1}{\eps^{2q}}|\Delta_{k+1}\mathcal{W}^j|^{2q}}.
\end{align*}
For any $j\in\na^\star$, the absolute moment of order $2q \in (2,\infty)$ of $\Delta_{k+1}\mathcal{W}^j$ can be computed explicitly as follows by using the integral definition of the expectation
\begin{align*}
\erww{|\Delta_{k+1}\mathcal{W}^j|^{2q}}
=\frac{2^q}{\sqrt{\pi}} \Gamma(q + \frac{1}{2})\tau ^q,
\end{align*}
where $\Gamma$ denotes the gamma function defined on $(0,+\infty)$ by $\Gamma: \alpha \mapsto \int_0^\infty t^{\alpha-1} e^{-t}\,dt$. 
Indeed, denoting by $X$ a random variable following the standard normal distribution $\mathcal{N}(0,1)$, we have for any $j\in\na^\star$, 
\begin{align*}
\erww{|\Delta_{k+1}\mathcal{W}^j|^{2q}}= \erwb |\sqrt{\tau} X|^{2q} \erwe=\tau^q \frac{\sqrt{2}}{\sqrt{\pi}} \int_0^{\infty} x^{2q} e^{\frac{-x^2}{2}}\,dx= \frac{2^q}{\sqrt{\pi}} \Gamma(q + \frac{1}{2})\tau^q.
\end{align*}
Then,
\begin{align*}
\erww{(G(\upsk)\Delta_{k+1}W, \pe(\upskp))_{L^2(D)}}\leq C_3 \erww{\|\upskp-\upsk\|_{L^2(D)}^2}+ C_4 \frac{\tau^q}{\eps^{2q}},
\end{align*}
with $C_3 := \frac{1}{2}(1+\frac{C_1(q-1)}{q}) >0$ and $C_4 := \frac{C_2}{2}\operatorname{Tr}(\mathcal{Q}) \frac{2^q}{\sqrt{\pi}} \Gamma(q + \frac{1}{2})\geq 0$. Summing over $k\in\{0,\ldots,N-1\}$ and using Proposition \ref{bounds} we obtain
\begin{align*}
\sum_{k=0}^{N-1} \erww{(G(\upsk)\Delta_{k+1}W, \pe(\upskp))_{L^2(D)}}
\leq C_3 K_0 + T C_4 \frac{\tau^{q-1}}{\eps^{2q}}
\leq C_3 K_0 + T C_4 \left(\frac{\tau}{\eps^{\frac{2q}{q-1}}}\right)^{q-1},
\end{align*}
which leads us to
\begin{align}\label{250813_06_ced}
\sum_{k=0}^{N-1} \erww{(G(\upsk)\Delta_{k+1}W, \pe(\upskp))_{L^2(D)}}
\leq C_3 K_0 + T C_4 \left( \frac{\tau}{\eps^{2+\theta}} \right)^{\frac{2}{\theta}}.
\end{align}
$\bullet$ Let us now turn to the study of the last term in \eqref{scmpeunpk_ced}. Using Young's inequality, the inequality $(a+b)^2\leq 2(a^2+b^2)$ and the Lipschitz continuity of $\beta$ stated in $(H_2)$, we get
\begin{align*}
(\beta(\upskp)+f_k,\pe(\upskp))_{L^2(D)} \leq L_{\beta}^2\| \upskp\|_{L^2(D)}^2+\| f_k\|_{L^2(D)}^2+\frac{1}{2}\| \pe(\upskp)\|_{L^2(D)}^2.
\end{align*}
Thus, by using the constant $K_0$ given by Proposition \ref{bounds}, we arrive at
\begin{align}\label{caro_10_02_a}
&\sum_{k=0}^{N-1} \tau \erww{(\beta(\upskp)+f_k,\pe(\upskp))_{L^2(D)}}
\leq C_5 + \frac{1}{2}\sum_{k=0}^{N-1}\tau \erww{\| \pe(\upskp)\|_{L^2(D)}^2},
\end{align}
with $C_5:=L_\beta^2TK_0+ \erww{\int_0^T \| f\|_{L^2(D)}^2\,dt}$. Gathering \eqref{T1_ced}, \eqref{T2_ced}, \eqref{250813_06_ced} and \eqref{caro_10_02_a}, we find that
\begin{align*}
\frac{1}{2}\sum_{k=0}^{N-1}\tau \erww{\| \pe(\upskp)\|_{L^2(D)}^2}
\leq C_3 K_0
+ T C_4 \left( \frac{\tau}{\eps^{2+\theta}} \right)^{\frac{2}{\theta}}
+ C_5
\end{align*}
and the result follows since by assumption $\tau=\mathcal{O}(\eps^{2+\theta})$.
\end{proof}
\begin{rem}
The condition $r<2$ in Assumption $(H_6)$ comes from the fact that in the previous proof, we want to apply Lemma \ref{lemma_G1_G3} for the exponent $\frac{4}{2+\theta}$, with $\theta>0$. To do so, we have to respect the condition $\frac{4}{2+\theta}\geq r$ which leads us to $\theta \leq \frac{4}{r}-2$, thus we must impose $\frac{4}{r}-2>0$, which led us to the condition $r<2$.
\end{rem}
\begin{lem}\label{pebounds}
Given the constant $r\in[1,2)$ of $(H_6)$, if we assume that there exists $\theta\in(0,\frac{4}{r}-2]$ such that $\frac{T}{N}=\mathcal{O}(\eps^{2+\theta})$, then the sequences
$(\pe(\uhnr))_{N,\eps}$ and $(\pe(\uhnl))_{N,\eps}$ are bounded in $L^2(\Omega;L^2(0,T;L^2(D)))$ and in $L^2_{\mathcal{P}_T}\big(\Omega_T;L^2(D)\big)$, respectively, independently of the discretization and regularization parameters $N\in\na^\star$ and $\eps\in\R_+^\star$, for $N$ sufficiently large.
\end{lem}
\begin{proof}
The boundedness of $(\pe(\uhnr))_{N,\eps}$ in $L^2(\Omega;L^2(0,T;L^2(D)))$ is a consequence of Proposition \ref{boundpe}. As for $(\pe(\uhnl))_{N,\eps}$, its boundedness in $L^2_{\mathcal{P}_T}\big(\Omega_T;L^2(D)\big)$ also comes from Proposition \ref{boundpe}, combined with the predictability of $\uhnl$ with values in $L^2(D)$ transmitted to $(\pe(\uhnl))_{N,\eps}$.
\end{proof}
\begin{lem}\label{250813_lem1}
There exists a constant $K_2 \geq 0$ not depending on $N\in\na^\star$ and on $\eps\in\R_+^\star$ such that, for $N$ sufficiently large,
\begin{align}\label{250813_lem1_eq}
&\erww{\max_{n\in\{0,\ldots,N\}}\|\upsn\|_{L^2(D)}^2} \leq K_2.
\end{align}
\end{lem}
\begin{proof}
In order to prove this lemma, we adapt arguments developed in \cite[Lemma 4]{VZ19}.\\
Set $N\in\na^\star$, $\tau=\frac{T}{N}$, $\eps\in\R_+^\star$, fix $n\in\{1,\ldots,N\}$ and assume that $\tau \leq \frac{3}{4} \frac{1}{1 + 2 L_\beta}$.\\
For any $k\in \{0,\ldots,n-1\}$, we multiply \eqref{240911_01} by $\upskp$ and integrate over $D$ to arrive at
\begin{align*}
&(\upskp-\upsk,\upskp)_{L^2(D)}
+ \tau \left(\plap\upskp,\upskp\right)_{L^2(D)}
+ \tau \left(\pe(\upskp),\upskp\right)_{L^2(D)}\\
=\,& \left(G(\upsk)\Delta_{k+1}W, \upskp\right)_{L^2(D)}
+ \tau \left(\beta(\upskp),\upskp\right)_{L^2(D)}
+ \tau \left(f_k,\upskp\right)_{L^2(D)},
\end{align*}
which can be decomposed as
\begin{align*}
I_1 + I_2 + I_3 = I_4 + I_5 + I_6.
\end{align*}
Let us study each term separately. Firstly, note that $I_1$ can be rewritten as
\begin{align}\label{ced1}
I_1 =\,& (\upskp-\upsk,\upskp)_{L^2(D)} \\
=\,& \frac{1}{2} \left( \| \upskp \|_{L^2(D)}^2 - \| \upsk \|_{L^2(D)}^2 + \| \upskp - \upsk \|_{L^2(D)}^2 \right).
\end{align}
Secondly, we remark that
\begin{align}\label{ced2}
I_2 =\,& \tau \left(\plap\upskp,\upskp\right)_{L^2(D)} = \tau \| \upskp \|_V^2 \geq 0.
\end{align}
Thirdly, thanks to the monotonicity of $\pe$ and the fact that $\pe(0) = 0$, we have
\begin{align}\label{ced3}
I_3 =\,& \tau \left(\pe(\upskp),\upskp\right)_{L^2(D)} = \tau \left(\pe(\upskp) - \pe(0),\upskp-0\right)_{L^2(D)} \geq 0.
\end{align}
Fourthly, by applying successively Cauchy-Schwarz inequality and Young's inequality, we prove that
\begin{align}\label{ced4}
I_4 =\,& \left(G(\upsk)\Delta_{k+1}W, \upskp\right)_{L^2(D)}\\
=\,& \left(G(\upsk)\Delta_{k+1}W, \upskp - \upsk\right)_{L^2(D)} + \left(G(\upsk)\Delta_{k+1}W, \upsk\right)_{L^2(D)}\\
\leq\,& \| G(\upsk)\Delta_{k+1}W \|_{L^2(D)} \| \upskp - \upsk\|_{L^2(D)} + (G(\upsk)\Delta_{k+1}W, \upsk)_{L^2(D)}\\
\leq\,& \frac{1}{2} \| G(\upsk)\Delta_{k+1}W \|_{L^2(D)}^2 + \frac{1}{2} \| \upskp - \upsk\|_{L^2(D)}^2
+ (G(\upsk)\Delta_{k+1}W, \upsk)_{L^2(D)}.
\end{align}
Fifthly, owing to the Lipschitz property of $\beta$ and the fact that $\beta(0) = 0$, we show that
\begin{align}\label{ced5}
\hspace*{-0.3cm}I_5 =\,& \tau \left(\beta(\upskp),\upskp\right)_{L^2(D)} = \tau \left(\beta(\upskp) - \beta(0),\upskp-0\right)_{L^2(D)}
\leq \tau L_\beta \| \upskp \|_{L^2(D)}^2.
\end{align}
At last, using Young's inequality we get
\begin{align}\label{ced6}
I_6 =\,& \tau \left(f_k,\upskp\right)_{L^2(D)} \leq \frac{\tau}{2} \| f_k \|_{L^2(D)}^2 + \frac{\tau}{2} \| \upskp \|_{L^2(D)}^2.
\end{align}
Gathering together (\ref{ced1})-(\ref{ced2})-(\ref{ced3})-(\ref{ced4})-(\ref{ced5})-(\ref{ced6}), discarding nonnegative terms, and multiplying by $2$, we arrive at
\begin{align*}
&\delta \| \upskp \|_{L^2(D)}^2 - \| \upsk \|_{L^2(D)}^2\\
\leq\,& \| G(\upsk)\Delta_{k+1}W \|_{L^2(D)}^2
+ 2 \left(G(\upsk)\Delta_{k+1}W, \upsk\right)_{L^2(D)}
+ \tau \| f_k \|_{L^2(D)}^2.
\end{align*}
where $\delta := 1 - \tau (1 + 2 L_\beta)$.
Since by assumption $0 < \tau \leq \frac{3}{4} \frac{1}{1 + 2 L_\beta}$, then $\frac{1}{4} \leq \delta <1$ and so $1 < \delta^{-1} \leq 4$, which leads us to
\begin{align*}
&\| \upskp \|_{L^2(D)}^2 - \| \upsk \|_{L^2(D)}^2 \\
\leq &
\ \delta^{-1} \left( \| G(\upsk)\Delta_{k+1}W \|_{L^2(D)}^2
+ 2 \left(G(\upsk)\Delta_{k+1}W, \upsk\right)_{L^2(D)}
+ \tau \| f_k \|_{L^2(D)}^2 \right)+ (\delta^{-1} - 1) \| \upsk \|_{L^2(D)}^2.
\end{align*}
Owing again to the assumption on $\tau$, we also have $0 < \delta ^{-1} - 1 = \delta^{-1} (1 + 2 L_\beta) \tau \leq 4 (1 + 2 L_\beta) \tau$ so
\begin{align*}
\| \upskp \|_{L^2(D)}^2 - \| \upsk \|_{L^2(D)}^2 \leq &4 \| G(\upsk)\Delta_{k+1}W \|_{L^2(D)}^2
+ 2 \delta^{-1} \left(G(\upsk)\Delta_{k+1}W, \upsk\right)_{L^2(D)}\\
&+ 4 \tau \| f_k \|_{L^2(D)}^2
+ 4 (1 + 2 L_\beta) \tau \| \upsk \|_{L^2(D)}^2,
\end{align*}
and after summing over $k \in \{0,\ldots,n-1\}$, we obtain
\begin{align*}
\| u_\eps^n \|_{L^2(D)}^2
\leq &\| u_\eps^0 \|_{L^2(D)}^2
+ 4 \sum_{k=0}^{n-1} \| G(\upsk)\Delta_{k+1}W \|_{L^2(D)}^2
+ 2\delta^{-1} \sum_{k=0}^{n-1} \left(G(\upsk)\Delta_{k+1}W, \upsk\right)_{L^2(D)}\\
&+ 4 \tau \sum_{k=0}^{n-1} \| f_k \|_{L^2(D)}^2
+ 4 (1 + 2 L_\beta) \tau \sum_{k=0}^{n-1} \| \upsk \|_{L^2(D)}^2.
\end{align*}
From the fact that $\delta^{-1} \leq 4$, one gets
\begin{align*}
\| \upsn \|_{L^2(D)}^2
\leq &\| u_\eps^0 \|_{L^2(D)}^2
+ 4 \sum_{k=0}^{n-1} \| G(\upsk)\Delta_{k+1}W \|_{L^2(D)}^2
+ 8 \left| \sum_{k=0}^{n-1} \left(G(\upsk)\Delta_{k+1}W, \upsk\right)_{L^2(D)} \right|\\
&+ 4 \tau \sum_{k=0}^{n-1} \| f_k \|_{L^2(D)}^2
+ 4 (1 + 2 L_\beta) \tau \sum_{k=0}^{n-1} \| \upsk \|_{L^2(D)}^2.
\end{align*}
Taking the maximum over $n \in \{1,\ldots,N\}$, we arrive at
\begin{align*}
\max_{n \in \{1,\ldots,N\}} \| \upsn \|_{L^2(D)}^2 & \leq\,
\| u_\eps^0 \|_{L^2(D)}^2
+ 4 \sum_{k=0}^{N-1} \| G(\upsk)\Delta_{k+1}W \|_{L^2(D)}^2\\
&+ 8 \max_{n \in \{1,\ldots,N\}} \left| \sum_{k=0}^{n-1} \left(G(\upsk)\Delta_{k+1}W, \upsk\right)_{L^2(D)} \right|\\
&+ 4 \tau \sum_{k=0}^{N-1} \| f_k \|_{L^2(D)}^2
+ 4 (1 + 2 L_\beta) \tau \sum_{k=0}^{N-1} \| \upsk \|_{L^2(D)}^2.
\end{align*}
Taking the expectation, we obtain
\begin{align}\label{ced7}
\erwb \max_{n \in \{1,\ldots,N\}} \| \upsn \|_{L^2(D)}^2 \erwe 
& \leq \erwb \| u_\eps^0 \|_{L^2(D)}^2 \erwe + J_1 + J_2 + J_3 + J_4.
\end{align}
where
\begin{align*}
J_1&=4 \sum_{k=0}^{N-1} \erwb \| G(\upsk)\Delta_{k+1}W \|_{L^2(D)}^2 \erwe\\
J_2&=8 \erwb \max_{n \in \{1,\ldots,N\}} \left| \sum_{k=0}^{n-1} \left(G(\upsk)\Delta_{k+1}W, \upsk\right)_{L^2(D)} \right| \erwe\\
J_3&=4 \erwb \sum_{k=0}^{N-1} \tau \| f_k \|_{L^2(D)}^2 \erwe\\
\text{and }J_4&=4 (1 + 2 L_\beta) \tau \sum_{k=0}^{N-1} \erwb \| \upsk \|_{L^2(D)}^2 \erwe.
\end{align*}
By $(H_4)$ and Itô's isometry \cite[Proposition 4.20, see also Remark 4.21 p. 97]{DPZ14}, we have
\begin{align*}
J_1 =\, 4 \erwb \int_0^T \| G(\uhnl(t)) \|_{\HS}^2\,dt \erwe
=\, & 4 \erwb \| G(\uhnl) \|_{L^2(0,T;\HS)}^2 \erwe.
\end{align*}
Now, we turn our attention to $J_2$:
\begin{align*}
 J_2 = & 8 \erwb \max_{n \in \{1,\ldots,N\}} \left| \sum_{k=0}^{n-1} \left(G(\upsk)\Delta_{k+1}W, \upsk\right)_{L^2(D)} \right| \erwe\\
= & 8 \erwb \max_{n \in \{1,\ldots,N\}} \left| \sum_{k=0}^{n-1} \left(\int_{t_k}^{t_{k+1}}G(\upsk)\,dW(s), \upsk\right)_{L^2(D)} \right| \erwe.
\end{align*}
Denoting the space of bounded linear operators from $L^2(D)$ to $\R$ by $\operatorname{Lin}(L^2(D);\R)$ and considering the mapping $\mathscr{L}:L^2(D)\to\operatorname{Lin}(L^2(D);\R)$ defined by $\mathscr{L}(\phi):v\in L^2(D)\mapsto (v,\phi)_{L^2(D)}\in \R$, for any $\phi\in L^2(D)$, the application of \cite[Lemma 2.4.1, p.42]{LR} with the linear operator $\mathscr{L}(\upsk)$ allows us to state that
\begin{equation*}
\left(\int_{t_k}^{t_{k+1}} G(\upsk)\,dW(s), \upsk\right)_{L^2(D)}=\int_{t_k}^{t_{k+1}}(G(\upsk)(\cdot),\upsk)_{L^2(D)}\,dW(s)
\end{equation*}
for any $k\in\{0,\ldots,n-1\}$. We obtain the following estimate for $J_2$:
\begin{align*}
J_2 = & 8 \erwb \max_{n \in \{1,\ldots,N\}} \left| \sum_{k=0}^{n-1} \int_{t_k}^{t_{k+1}}\left(G(\upsk)(\cdot), \upsk\right)_{L^2(D)}\,dW(s) \right| \erwe\\
= & 8 \erwb \max_{n \in \{1,\ldots,N\}} \left| \int_0^{t_n} \left(G(\uhnl(s))(\cdot), \uhnl(s)\right)_{L^2(D)}\,dW(s) \right| \erwe \\
= & 8 \erwb \max_{n \in \{1,\ldots,N\}} \left| \int_0^{t_n} \left(G(\uhnl(s))(\cdot), \uhnl(s)\right)_{L^2(D)}\,dW(s) \right| \erwe\\
= & 8\erwb \sup_{t \in [0,T]} \left| \int_0^{t} \left(G(\uhnl(s))(\cdot), \uhnl(s)\right)_{L^2(D)}\,dW(s) \right| \erwe.
\end{align*}
From \cite[Theorem 4.36, p.114]{DPZ14}, we may deduce that
\begin{align*}
&J_2
\leq 24 \left(\erwb\int_0^T \| \left(G(\uhnl(s))(\cdot), \uhnl(s)\right)_{L^2(D)} \|_{HS(L^2(D);\R)}^2\,ds \erwe\right)^{\frac{1}{2}}.
\end{align*}
Noting that for all $s\in [0,T]$, $\left(G(\uhnl(s))(\cdot),\uhnl(s)\right)_{L^2(D)} = \mathscr{L}(\uhnl(s))\big( G(\uhnl(s))\big)$,
we obtain from standard estimates for Hilbert-Schmidt operators (see, \textit{e.g.,} \cite[Remark B.0.6, p.217]{LR}) and Riesz isomorphism, that
\begin{align*}
\| \left(G(\uhnl(s))(\cdot), \uhnl(s)\right)_{L^2(D)} \|_{HS(L^2(D);\R)}^2
&\leq \| G(\uhnl)(s) \|_{\HS}^2 \Vert \mathscr{L}(\uhnl(s)) \Vert^2_{\operatorname{Lin}(L^2(D);\R)}\\
&\leq \| G(\uhnl)(s) \|_{\HS}^2 \Vert \uhnl(s) \Vert^2_{L^2(D)}\\
&\leq | G(\uhnl)(s) \|_{\HS}^2 \sup_{t\in [0,T]} \| \uhnl(t) \|_{L^2(D)}^2.
\end{align*}
Consequently,
\begin{align*}
J_2&\leq 24 \erwb \left(\int_0^T \| G(\uhnl)(s) \|_{\HS}^2\,ds \sup_{t\in[0,T]} \|\uhnl(t)\|_{L^2(D)}^2 \right)^{\frac{1}{2}} \erwe\\
&\leq 24 \erwb \| G(\uhnl) \|_{L^2(0,T;\HS)} \sup_{t\in [0,T]} \|\uhnl(t)\|_{L^2(D)} \erwe.
\end{align*}
Applying Young's inequality with any $\alpha > 0$, we arrive at
\begin{align*}
J_2 \leq\ & \frac{12}{\alpha} \erwb \| G(\uhnl) \|_{L^2(0,T;\HS)}^2 \erwe
+ 12 \alpha \erwb \sup_{t\in[0,T]} \|\uhnl(t)\|_{L^2(D)}^2 \erwe\\
\leq\ & \frac{12}{\alpha} \erwb \| G(\uhnl) \|_{L^2(0,T;\HS)}^2 \erwe
+ 12 \alpha \erwb \max_{n\in\{1,\ldots,N\}} \|\upsn\|_{L^2(D)}^2 \erwe\\
&+ 12 \alpha \erwb \|u_0\|_{L^2(D)}^2 \erwe.
\end{align*}
With Lemma \ref{boundguhnlr_and_betauhnlr} and Hypothesis $(H_1)$, there exist constants $C_1, C_2>0$ that are independent of the regularization and discretization parameters and such that
\begin{align}\label{J1}
J_1 \leq C_1
\end{align}
and
\begin{align}\label{J2}
&J_2
\leq C_2
+ 12 \alpha \erwb \max_{n\in\{1,\ldots,N\}} \|\upsn\|_{L^2(D)}^2 \erwe.
\end{align}
From Lemma \ref{CVfhnl}, there exists a constant $C_3>0$ independent of the regularization and discretization parameters and such that
\begin{align}\label{J3}
J_3 = 4 \erwb \sum_{k=0}^{N-1} \tau \| f_k \|_{L^2(D)}^2 \erwe
= 4 \| f^N \|_{L^2(\Omega;L^2(0,T;L^2(D)))}^2
\leq C_3.
\end{align}
Thanks to the estimate \eqref{uhnboundbis},
\begin{align}\label{J4}
J_4 \leq 4 (1 + 2 L_\beta) \tau N \max_{n\in\{1,\ldots,N\}} \erwb \| \upsk \|_{L^2(D)}^2 \erwe
\leq 4 (1 + 2 L_\beta) T \Upsilon.
\end{align}
Plugging \eqref{J1}, \eqref{J2}, \eqref{J3} and \eqref{J4} into \eqref{ced7}, we get
\begin{align*}
\erwb \max_{n \in \{1,\ldots,N\}} \| \upsn \|_{L^2(D)}^2 \erwe \leq\ &
\erwb \| u_\eps^0 \|_{L^2(D)}^2 \erwe
+ C_1 + C_2 + C_3 + 4 (1 + 2 L_\beta) T \Upsilon
+ 12 \alpha \erwb \max_{n\in\{1,\ldots,N\}} \|\upsn\|_{L^2(D)}^2 \erwe.
\end{align*}
For convenience, we set $\alpha = 1 / 24$ in this last inequality, multiply both sides by $2$, and we have the announced result thanks to Hypothesis $(H_1)$:
\begin{align*}
\erwb \max_{n \in \{1,\ldots,N\}} \| \upsn \|_{L^2(D)}^2 \erwe & \leq\,
2 \bigg( \| u_0 \|_{L^2(\Omega;L^2(D))}^2 + C_1 + C_2 + C_3 \bigg)
+ 8 (1 + 2 L_\beta) T \Upsilon.
\end{align*}
\end{proof}
The bound \eqref{250813_lem1_eq} given by Lemma \ref{250813_lem1} is crucial to get the following estimate on the sequence $(M_{N,\eps}-\widehat{M}_{N,\eps})_{N,\eps}$ which will ensure that if the sequences $(M_{N,\eps})_{N,\eps}$ and $(\widehat{M}_{N,\eps})_{N,\eps}$ weakly converge, it is necessarily towards the same limit:
\begin{lem}\label{251007_l1}
There exists a constant $\gamma>0$ not depending on $N\in\na^\star$ and on $\eps\in\R_+^\star$ such that, for $N$ sufficiently large,
\begin{align}\label{160316_101}
&\erww{\sup_{t\in[0,T]}\left\| M_{N,\eps}(t)-\widehat{M}_{N,\eps}(t)\right\|_{L^2(D)}}
\leq K_2 \tau^\gamma,
\end{align}
where $K_2\geq 0$ is a constant given by Lemma \ref{250813_lem1}.
\end{lem}
\begin{proof}
We fix $N\in\na^\star$, $\eps\in\R_+^\star$, and recall that $\widehat{M}_{N,\eps}(T) = M_{N,\eps}(T)$. We have
\begin{align*}
&\sup_{t\in[0,T]}\left\| M_{N,\eps}(t)-\widehat{M}_{N,\eps}(t)\right\|_{L^2(D)}\\
=\ &\sup_{n\in\{0,\ldots,N-1\}}\sup_{t\in [t_n,t_{n+1}]}\left\|\int_{t_n}^t G(\upsn)\,dW(s) - \frac{(t-t_n)}{\tau}\int_{t_n}^{t_{n+1}}G(\upsn)\,dW(s)\right\|_{L^2(D)}\\
\leq\ & \sup_{n\in\{0,\ldots,N-1\}}\sup_{t\in [t_n,t_{n+1}]}\left(\left\|\int_{t_n}^{t} G(\upsn)\,dW(s)\right\|_{L^2(D)}+\left\|\int_{t_n}^{t_{n+1}} G(\upsn)\,dW(s)\right\|_{L^2(D)}\right).
\end{align*}
By remarking that 
\begin{align*}
\left\|\int_{t_n}^{t_{n+1}} G(\upsn)\,dW(s)\right\|_{L^2(D)}
\leq \sup_{t\in [t_n,t_{n+1}]} \left\|\int_{t_n}^{t} G(\upsn)\,dW(s)\right\|_{L^2(D)},
\end{align*}
we have
\begin{align*}
&\sup_{t\in [0,T]}\left\| M_{N,\eps}(t)-\widehat{M}_{N,\eps}(t)\right\|_{L^2(D)}
\leq 2\sup_{n\in\{0,\ldots,N-1\}}\sup_{t\in [t_n,t_{n+1})} \left\|\int_{t_n}^{t} G(\upsn)\,dW(s)\right\|_{L^2(D)}.
\end{align*}
Now, for $n\in \{0,\ldots,N-1\}$ and $t\in [t_n,t_{n+1}]$, 
\begin{align}\label{070426_bc}
&\left\|\int_{t_n}^{t} G(\upsn)\,dW(s)\right\|_{L^2(D)}\leq C \left\| G(\upsn) \right\|_{\HS}\left\| W(t) - W(t_n) \right\|_U
\end{align}
where $C\geq 0$ is a constant that may change from line to line. Moreover, using the Garsia-Rodemich-Rumsey inequality (\cite{GRRR70} and \cite[Exercice 2.4.1]{SV79}) we obtain for any $q\geq1$ and $\alpha > 1/q$,
\begin{align}\label{070426_cc}
&\left\| W(t) - W(t_n) \right\|_U\leq C \tau^{(\alpha -1/q)}X^{\frac{1}{q}}
\end{align}
where $X := \int_0^T\int_0^T \left( \left\| W(s) - W(t_n) \right\|_U^q / |s-z|^{\alpha q + 1} \right)\,ds\,dz$ and $\gamma:=(\alpha -1/q)>0$. Plugging \eqref{070426_cc} into \eqref{070426_bc}, we get
\begin{align}\label{07426_dc}
 &\sup_{n\in\{0,\ldots,N-1\}}\sup_{t\in [t_n,t_{n+1})} \left\|\int_{t_n}^{t} G(\upsn)\,dW(s)\right\|_{L^2(D)}\nonumber\\
 \leq\ & C \sup_{t\in[0,T]} \left\| G(\uhnl(t)) \right\|_{\HS}
\tau^{\gamma} X^{1/q}\nonumber\\
\leq\ & C \max_{n\in \{0,\ldots,N-1\}} \left\| u^n_\eps \right\|_{L^2(D)}\tau^{\gamma} X^{1/q}.
\end{align}
In this manner, taking the expectation in \eqref{07426_dc} and using H\"older's inequality, one gets
\begin{align}\label{260409_01}
\begin{aligned}
&\erwb\sup_{n\in\{0,\ldots,N-1\}}\sup_{t\in [t_n,t_{n+1})} \left\|\int_{t_n}^{t} G(\upsn)\,dW(s)\right\|_{L^2(D)} \erwe\\
\leq \ & C\tau^{\gamma} \erwb \max_{n\in \{0,\ldots,N-1\}} \left\| u^n_\eps \right\|_{L^2(D)}
 X^{1/q}\erwe\\
 \leq \ & C\tau^{\gamma}\erww{\max_{n\in \{0,\ldots,N-1\}} \left\| u^n_\eps \right\|_{L^2(D)}^2}^{1/2}\erww{X^{2/q}}^{1/2}.
 \end{aligned}
 \end{align}
Choosing $q>2$ and using Jensen's inequality, one arrives at
\begin{align}\label{07426_ec}
\erwb X^{2/q}\erwe \leq \left(\erwb X\erwe\right)^{2/q} \leq \left(C \int_0^T\int_0^T |t-r|^{q/2-\alpha q -1}dtdr\right)^{2/q}<+\infty.
\end{align}
Now, the assertion follows by plugging into \eqref{260409_01} the inequality \eqref{07426_ec} and the uniform bound \eqref{250813_lem1_eq} stated in Lemma \ref{250813_lem1}.
\end{proof}
\begin{lem}\label{260409_lem1}
There exists a constant $K_3\geq 0$ not depending on $N\in\na^\star$ and on $\eps\in\R_+^\star$ such that, for $N$ sufficiently large,
\begin{align*}
\erww{\sup_{t\in[0,T]}\left\| M_{N,\eps}(t)-\widehat{M}_{N,\eps}(t)\right\|_{L^2(D)}^2}
\leq K_3.
\end{align*}
\end{lem}
\begin{proof}
Similarly as in the proof of Lemma \ref{251007_l1} we obtain
\begin{align*}
&\erww{\sup_{t\in[0,T]}\left\| M_{N,\eps}(t)-\widehat{M}_{N,\eps}(t)\right\|_{L^2(D)}^2}\\
\leq\ & 4\erww{\sup_{n\in\{0,\ldots,N-1\}}\sup_{t\in [t_n,t_{n+1}]}\left\Vert \int_{t_n}^t G(\upsn)\,dW(s)\right\Vert_{L^2(D)}^2}
\end{align*}
By integral splitting and using Burkholder-Davis-Gundy inequality in the last step, we find a constant $C>0$ such that
\begin{align*}
&\erww{\sup_{t\in[0,T]}\left\| M_{N,\eps}(t)-\widehat{M}_{N,\eps}(t)\right\|_{L^2(D)}^2}\\
&\leq 8\erww{\sup_{n\in\{0,\ldots,N-1\}}\sup_{t\in [t_n,t_{n+1}]}\left\Vert\int_0^t G(\upsn)\,dW(s)\right\Vert_{L^2(D)}^2+\sup_{n\in\{0,\ldots,N-1\}}\left\Vert\int_0^{t_n} G(\upsn)\,dW(s)\right\Vert_{L^2(D)}^2}\\
&\leq 16\erww{\sup_{t\in [0,T]}\left\Vert\int_0^t G(\uhnl)\,dW(s)\right\Vert_{L^2(D)}^2}\leq C\erww{\int_0^T \left\Vert G(\upsn)\right\Vert_{\HS}^2\,ds}.
\end{align*}
Now, the assertion follows from Lemma \ref{boundguhnlr_and_betauhnlr}.
\end{proof}
\begin{lem}\label{251016_l1}
Given the constant $r\in[1,2)$ of $(H_6)$, if we assume that there exists $\theta\in(0,\frac{4}{r}-2]$ such that $\frac{T}{N}=\mathcal{O}(\eps^{2+\theta})$, then the sequence
$(\partial_t\left(\widehat{u}_{N,\eps}-\widehat{M}_{N,\eps}\right))_{N,\eps}$ is bounded in $L^{p'}(\Omega;L^{p'}(0,T;V^*))$, independently of the discretization and regularization parameters $N\in\na^\star$ and $\eps\in\R_+^\star$, for $N$ sufficiently large.
\end{lem}
\begin{proof}
We adapt here the arguments of \cite[Lemma 5]{VZ19} (see also \cite[Lemma 2.5]{VZ21}).
Let $n \in \{0,\ldots,N-1\}$. From \eqref{240911_01} the following time discrete equation holds in $V^*$:
for all $t \in [t_n, t_{n+1})$,
\begin{align*}
\partial_t\left(\widehat{u}_{N,\eps}-\widehat{M}_{N,\eps}\right)(t)
= \frac{\upsnp - \upsn - G(\upsn) \Delta_{n+1} W}{\tau}
= - \plap \upsnp - \pe(\upsnp) + \beta(\upsnp) + f_n.
\end{align*}
Hence on $[0,T)$, we have
\begin{align*}
\partial_t\left(\widehat{u}_{N,\eps}-\widehat{M}_{N,\eps}\right)
& = \sum_{n=0}^{N-1} \left( - \plap \upsnp - \pe(\upsnp) + \beta(\upsnp) + f_n \right) \mathds{1}_{[t_n,t_{n+1})}\\
& = - \plap \uhnr - \pe(\uhnr) + \beta(\uhnr) + f^N.
\end{align*}
By using the triangle inequality of the $L^{p'}(\Omega;L^{p'}(0,T;V^*))$ norm, we arrive at
\begin{align}\label{triangle_ineq}
&\left\| \partial_t\left(\widehat{u}_{N,\eps}-\widehat{M}_{N,\eps}\right)\right\|_{L^{p'}(\Omega;L^{p'}(0,T;V^*))}^{p'}\nonumber\\
 =\ & \erwb \int_0^T \left\| - \plap \uhnr(t) - \pe(\uhnr(t)) + \beta(\uhnr(t)) + f^N(t) \right\|_{V^*}^{p'}\,dt \erwe\nonumber\\
 \leq\ & \erwb \int_0^T \left\| \plap \uhnr(t) \right\|_{V^*}^{p'}\,dt \erwe
+ \erwb \int_0^T \left\| \pe(\uhnr(t)) \right\|_{V^*}^{p'}\,dt \erwe\nonumber\\
& + \erwb \int_0^T \left\| \beta(\uhnr(t)) \right\|_{V^*}^{p'}\,dt \erwe
+ \erwb \int_0^T \left\| f^N(t) \right\|_{V^*}^{p'}\,dt \erwe.
\end{align}
For the $p$-Laplace term, we use the same method as in the proof of \cite[Lemma 5]{VZ19}.
For all $n \in \{0,\ldots,N-1\}$, by definition we have
\begin{align*}
&\| \plap \upsnp \|_{V^*} = \sup_{\phi \in V,\|\phi\|_V\leq1}
\left| \langle \plap \upsnp, \phi \rangle_{V^*,V} \right|\\
=& \sup_{\phi \in V,\|\phi\|_V\leq1}
\left| \int_D \Big(
|\nabla \upsnp(x)|^{p-2} \nabla \upsnp(x) \cdot \nabla \phi(x)
+ |\upsnp(x)|^{p-2} \upsnp(x) \phi(x) \Big)
\,dx \right|\\
\leq& \sup_{\phi \in V,\|\phi\|_V\leq1}
\left| \int_D |\nabla \upsnp(x)|^{p-2} \nabla \upsnp(x) \cdot \nabla \phi(x)\,dx \right|\nonumber\\
&+ \sup_{\phi \in V,\|\phi\|_V\leq1}
\left| \int_D |\upsnp(x)|^{p-2} \upsnp(x) \phi(x)\,dx \right|.
\end{align*}
Let $\phi \in V$ be such that $\|\phi\|_V\leq1$. By applying H\"older's inequality, we obtain
\begin{align*}
\left|\int_D |\nabla \upsnp(x)|^{p-2} \nabla \upsnp(x) \cdot \nabla \phi(x)\,dx\right|
& \leq\int_D |\nabla \upsnp(x)|^{p-1} |\nabla \phi(x)|\,dx\\
& \leq \big\| |\nabla \upsnp|^{p-1} \big\|_{L^{p'}(D)} \| \nabla \phi \|_{\left(L^p(D)\right)^d}\\
& \leq \big\| \nabla \upsnp \|_{\left(L^p(D)\right)^d}^{p-1} \| \nabla \phi \|_{\left(L^p(D)\right)^d}.
\end{align*}
Similarly, we have
\begin{align*}
\left|\int_D |\upsnp(x)|^{p-2} \upsnp(x) \phi(x)\,dx\right|
& \leq \| \upsnp \|_{L^p(D)}^{p-1} \| \phi \|_{L^p(D)},
\end{align*}
and so
\begin{align*}
\| \plap \upsnp \|_{V^*}
& \leq \sup_{\phi \in V,\|\phi\|_V\leq1}
\| \nabla \upsnp \|_{\left(L^p(D)\right)^d}^{p-1} \| \nabla \phi \|_{\left(L^p(D)\right)^d}
+ \sup_{\phi \in V,\|\phi\|_V\leq1}
\| \upsnp \|_{L^p(D)}^{p-1} \| \phi \|_{L^p(D)}\\
& \leq \| \nabla \upsnp \|_{\left(L^p(D)\right)^d}^{p-1}
+ \| \upsnp \|_{L^p(D)}^{p-1}\\
&\leq 2 \| \upsnp \|_V^{p-1}.
\end{align*}
Then
$\| \plap \upsnp \|_{V^*}^{p'} \leq 2^{p'} \| \upsnp \|_V^p$
and finally, using the fact that $p'(p-1)=p$, we get
\begin{align}\label{plap_ineq}
\erwb \int_0^T \left\| \plap \uhnr(t) \right\|_{V^*}^{p'}\,dt \erwe
& = \erwb \sum_{n=0}^{N-1} \int_{t_n}^{t_{n+1}} \left\| \plap \upsnp \right\|_{V^*}^{p'}\,dt \erwe\nonumber\\
&\leq 2^{p'} \erwb \sum_{n=0}^{N-1} \int_{t_n}^{t_{n+1}} \| \upsnp \|_V^p\,dt \erwe\nonumber\\
& \leq 2^{p'} \| \uhnr \|_{L^p(\Omega;L^p(0,T;V))}^p.
\end{align}
For the other terms, related to $\pe(\uhnr)$, $\beta(\uhnr)$ and $f^N$, we use a common argument. The Gelfand triple $(V,L^2(D),V^*)$ gives existence to a constant $C_{V^*} > 0$ such that for any $v$ in $L^2(\Omega_T;L^2(D))$,
\begin{align*}
\erwb \int_0^T \left\| v(t) \right\|_{V^*}^{p'}\,dt \erwe
\leq C_{V^*}^{p'} \erwb \int_0^T \left\| v(t) \right\|_{L^2(D)}^{p'}\,dt \erwe.
\end{align*}
Note that since $p \geq 2$, then $1 < p' \leq 2$, and thus for all $v$ in $L^2(\Omega_T;L^2(D))$,
\begin{align}\label{gelfand_ineq}
\erwb \int_0^T \left\| v(t) \right\|_{V^*}^{p'}\,dt \erwe
\leq & C_{V^*}^{p'}\erwb \int_0^T \left\| v(t) \right\|_{L^2(D)}^{p'}\,dt \erwe\nonumber\\ 
=& C_{V^*}^{p'}\int_{\Omega}\int_0^T \left\| v(\omega,t) \right\|_{L^2(D)}^{p'} \mathds{1}_{\left\{(\omega,t)\in \Omega\times (0,T) : \left\| v(\omega, t) \right\|_{L^2(D)}\geq 1\right\}}(\omega,t)\,dt\,d\mathds{P}(\omega)\nonumber\\
&+C_{V^*}^{p'}\int_{\Omega}\int_0^T \left\| v(\omega,t) \right\|_{L^2(D)}^{p'} \mathds{1}_{\left\{(\omega,t)\in \Omega\times (0,T) : \left\| v(\omega,t) \right\|_{L^2(D)}< 1\right\}}(\omega,t)\,dt\,d\mathds{P}(\omega)\nonumber\\
\leq & C_{V^*}^{p'} \left(\left\| v \right\|_{L^2(\Omega_T;L^2(D))}^2+T\right).
\end{align}
By combining \eqref{plap_ineq} with the inequality \eqref{gelfand_ineq} applied independently to $\pe(\uhnr)$, $\beta(\uhnr)$ and $f^N$, we arrive from \eqref{triangle_ineq} to
\begin{align*}
\left\| \partial_t\left(\widehat{u}_{N,\eps}-\widehat{M}_{N,\eps}\right)\right\|_{L^{p'}(\Omega;L^{p'}(0,T;V^*))}^{p'}
\leq& \ 2^{p'} \| \uhnr \|_{L^p(\Omega;L^p(0,T;V))}^p
+ C_{V^*}^{p'} \left\| \pe(\uhnr) \right\|_{L^2(\Omega_T;L^2(D))}^2\\
& + C_{V^*}^{p'} \left(
\left\| \beta(\uhnr) \right\|_{L^2(\Omega_T;L^2(D))}^2
+ \left\| f^N \right\|_{L^2(\Omega_T;L^2(D))}^2
+ 3 T
\right)
\end{align*}
and the proof is complete with the boundedness results shown in Lemmas \ref{210611_lem01}, \ref{boundguhnlr_and_betauhnlr}, \ref{CVfhnl} and \ref{pebounds}.
\end{proof}
\begin{lem}\label{buhnrppanp}
Given the constant $r\in[1,2)$ of $(H_6)$, if we assume that there exists \\ $\theta\in(0,\frac{4}{r}-2]$ such that $\frac{T}{N}=\mathcal{O}(\eps^{2+\theta})$, then the sequences
$\left(\frac{(\uhnr)^-}{\eps}\right)_{N,\eps}$ and $\left(\frac{(\uhnr-1)^+}{\eps}\right)_{N,\eps}$ are bounded in $L^2(\Omega;L^2(0,T;L^2(D)))$, independently of the discretization and regularization parameters $N\in\na^\star$ and $\eps\in\R_+^\star$, for $N$ sufficiently large.
\end{lem}
\begin{proof} Since $\di\pe(\uhnr)=-\frac{(\uhnr)^-}{\eps}+\frac{(\uhnr-1)^{+}}{\eps}$ and $(\uhnr)^-\times(\uhnr-1)^+=0$, one gets that
\begin{equation*}
\left\|-\frac{(\uhnr)^-}{\eps} \right\|_{L^2(\Omega;L^2(0,T;L^2(D)))}^2+\left\|\frac{(\uhnr-1)^+}{\eps} \right\|_{L^2(\Omega;L^2(0,T;L^2(D)))}^2\\
=\big\|\pe(\uhnr)\big\|_{L^2(\Omega;L^2(0,T;L^2(D)))}^2,
\end{equation*}
As by Lemma \ref{pebounds} the right-hand side is bounded, the final result holds.
\end{proof}
One can prove the following, with the same method:
\begin{lem}\label{buhnlppanp}
Given the constant $r\in[1,2)$ of $(H_6)$, if we assume that there exists \\ $\theta\in(0,\frac{4}{r}-2]$ such that $\frac{T}{N}=\mathcal{O}(\eps^{2+\theta})$, then the sequences
$\left(\frac{(\uhnl)^-}{\eps}\right)_{N,\eps}$ and $\left(\frac{(\uhnl-1)^+}{\eps}\right)_{N,\eps}$ are bounded in $L^2_{\mathcal{P}_T}\big(\Omega_T;L^2(D)\big)$ independently of the discretization and regularization parameters $N\in\na^\star$ and $\eps\in\R_+^\star$, for $N$ sufficiently large.
\end{lem}
\begin{rem}\label{remlimcomu}
Using Lemma \ref{buhnrppanp}, Lemma \ref{buhnlppanp} and by expanding the square term of \eqref{controldiffpe(uhnl)pe(uhnr)}, one can prove that
\begin{align*}
\left\|\frac{(\uhnr)^-}{\eps}-\frac{(\uhnl)^-}{\eps} \right\|_{L^2(\Omega;L^2(0,T;L^2(D)))}^2
&\leq \frac{\tau}{\eps^2}K_0\\
\text{ and } \left\|\frac{(\uhnr-1)^+}{\eps}-\frac{(\uhnl-1)^+}{\eps} \right\|_{L^2(\Omega;L^2(0,T;L^2(D)))}^2
&\leq \frac{\tau}{\eps^2}K_0,
\end{align*}
which ensures that if the sequences $\left(\frac{(\uhnr)^-}{\eps}\right)_{\eps}$ and $\left(\frac{(\uhnl)^-}{\eps}\right)_{\eps}$ (respectively $\left(\frac{(\uhnr-1)^+}{\eps}\right)_\eps$ and $\left(\frac{(\uhnl-1)^+}{\eps}\right)_\eps$) converge strongly or weakly in $L^2(\Omega;L^2(0,T;L^2(D)))$, it is necessarily towards a common limit.
\end{rem}
\section{Proof of Theorem \ref{main}}\label{Section4}
In this section, owing to the stability estimates proved in the previous section, we will show existence of weak limits for the sequences listed in \eqref{040526_a}. Thanks to these weak convergence results, we will be able to pass to the limit in our time-discretization scheme and prove that such limits satisfy a \enquote{limit equation}. Then, the idenfication of weak limits coming from the discretization of non-linear terms of Problem\eqref{equation} will allow us to conclude that there exists a pair of solutions for Problem\eqref{equation} in the sense of Definition \ref{solution}. Finally, the remaining of this section will be then devoted to the proof of the uniqueness of such a pair of solution.\\

In what follows, let $(N_m)_{m\in\na} \subset\na^\star$ be a sequence with $\lim_{m\to\infty} N_m=+\infty$, set $\tau_m:=\frac{T}{N_m}$, and let $(\eps_m)_{m\in\na}\subset\R_+^\star$ be another sequence such that $\lim_{m\to\infty} \eps_m=0$. Moreover, given the constant $r\in[1,2)$ of $(H_6)$, for any $\theta \in (0,\frac{4}{r}-2]$, assume that for any $m\in\na$, $\tau_m=\mathcal{O}(\eps_m^{2+\theta})$. For any $m\in\na$, when the $m$-dependence is not required for the understanding of the main arguments, we shall use the notations $\tau:=\tau_m$, $N:=N_m$ and $\eps:=\eps_m$.
\subsection{Existence of weak limits}
\begin{lem}\label{addreg u}
There exists $u\in L^p_{\mathcal{P}_T}\big(\Omega_T;V\big)$ such that, up to subsequences denoted in the same way, the sequences $(\uhnr)_m$, $(\uhnl)_m$ and $(\widehat{u}_{N,\eps})_m$ weakly converge towards $u$ in $L^p(\Omega_T;V)$, $L^2_{\mathcal{P}_T}(\Omega_T;L^2(D))$, and in $L^2(\Omega_T;L^2(D))$, respectively as $m\to\infty$.
\end{lem}
\begin{proof} 
The convergence in $L^2(\Omega_T;L^2(D))$ to possibly distinct limits comes from the \textit{a priori} bounds on the discrete solutions of Lemma \ref{210611_lem01}. From \eqref{210824_05} it follows that the weak limits of $(\uhnr)_m$ and $(\uhnl)_m$ must coincide. Then, from the boundedness of $(\nabla \uhnr)_m$ in $L^p(\Omega_T;(L^p(D))^d)$ and the boundedness of $(\uhnr)_m$ in $L^p(\Omega_T;L^p(D))$ provided by Lemma \ref{210611_lem01}, the rest of the convergence results hold true. We remark that since $p>2$, $L^p(\Omega_T;V)$ is continuously embedded into $L^2(\Omega_T;L^2(D))$.
The predictability of $u$ is given \textit{a priori} with values in $L^2(D)$. Since this implies weak measurability with respect to $\mathcal{P}_T$ with values in $V$, by the Pettis measurability theorem (see, \textit{e.g.,} \cite[Theorem, 1.1.6 and Corollary 1.1.8]{HNVW16}) and the separability of $V$ it follows that $u$ is strongly measurable with respect to $\mathcal{P}_T$ with values in $V$.
It is an immediate consequence of Proposition \ref{bounds} (see, \textit{e.g.,} \cite[Lemma 7]{VZ19}) combined with Lemma \ref{210611_lem01}, that $(\widehat{u}_{N,\eps})_m$ defined in Definition \ref{piecewiseconstantbis} also weakly converges in $L^2(\Omega_T;L^2(D))$ towards $u$.
\end{proof}
\begin{lem}\label{251006_l1}
There exist $\vec{z}\in L^{p'}(\Omega_T;(L^{p'}(D))^d)$ and $y\in L^{p'}(\Omega_T;L^{p'}(D))$ such that, up to not relabeled subsequences, $(|\nabla \uhnr|^{p-2}\nabla \uhnr)_m$ converges to $\vec{z}$ weakly in $L^{p'}(\Omega_T;(L^{p'}(D))^d)$, and $(|\uhnr|^{p-2}\uhnr)_m$ converges to $y$ weakly in $L^{p'}(\Omega_T;L^{p'}(D))$, as $m\to\infty$.
\end{lem}
\begin{proof}
By recalling that
\begin{align*}
\left\||\nabla\uhnr|^{p-2}\nabla\uhnr\right\|_{L^{p'}(\Omega_T;(L^{p'}(D))^d)}^{p'} &= \|\nabla \uhnr\|_{L^p(\Omega_T;(L^p(D))^d)}^p\\
\text{and } \left\| |\uhnr|^{p-2}\uhnr \right\|_{L^{p'}(\Omega_T;L^{p'}(D))}^{p'} &= \| \uhnr \|_{L^p(\Omega_T;L^p(D))}^p,
\end{align*}
the existence of $\vec{z}$ and $y$ follows from the \textit{a priori} bounds on the $L^p$-norms of $(\nabla \uhnr)_m$ and $(\uhnr)_m$ provided by Lemma \ref{210611_lem01}.
\end{proof}
\begin{lem}\label{251016_l3}
For any $t\in [0,T]$, there exists an element $\chi_t\in L^2(\Omega\times D)$, and a subsequence $(\widehat{u}_{N,\eps}(t))_{m_t}$ of $(\widehat{u}_{N,\eps}(t))_{m}$, that may \textit{a priori} depend on $t$, such that $(\widehat{u}_{N,\eps}(t))_{m_t}$ converges to $\chi_t$ weakly in $L^2(\Omega\times D)$ as $m_t\to\infty$.
\end{lem}
\begin{proof}
This is a direct consequence of Proposition \ref{bounds}.
\end{proof}
\begin{lem}\label{CVpeuhnr}
There exists a process $\psi$ in $L^2_{\mathcal{P}_T}\big(\Omega_T;L^2(D)\big)$ such that, up to subsequences denoted in the same way, $(\pe(\uhnr))_m$ and $(\pe(\uhnl))_m$ weakly converge towards $\psi$ in $L^2(\Omega;L^2(0,T;L^2(D)))$ as $m\to\infty.$
\end{lem}
\begin{proof}
This comes as an immediate consequence of Remark \ref{251002_r1} and Proposition \ref{boundpe}. The limit process $\psi$ is predictable with values in $L^2(D)$ because at the limit $(\pe(\uhnl))_m$ is.
\end{proof}
\begin{lem}\label{CVnpapp} There exist $\psi_1,\psi_2$ in $L^2_{\mathcal{P}_T}\big(\Omega_T;L^2(D)\big)$ such that, up to subsequences denoted in the same way, $\left(-\frac{(\uhnr)^-}{\eps}\right)_{m}$ converges towards $\psi_1$ and $\left(\frac{(\uhnr-1)^+}{\eps}\right)_{m}$ converges towards $\psi_2$, both weakly in $L^2(\Omega;L^2(0,T;L^2(D)))$ as $m\to\infty.$
\end{lem}
\begin{proof}
The result holds true thanks to Lemma \ref{buhnrppanp}, Lemma \ref{buhnlppanp} and Remark \ref{remlimcomu}.
\end{proof}
\begin{lem}\label{CVSppapn}
The following convergences hold strongly in $L^2(\Omega;L^2(0,T;L^2(D)))$:
\begin{equation*}
(\uhnr)^-\to 0 \text{ and }(\uhnr-1)^+\to 0 \text{ as }m\to\infty.
\end{equation*}
\end{lem}
\begin{proof}From Lemma \ref{buhnrppanp}, there exists a constant $M>0$ not depending on the parameters $N\in\na^\star$ and $\eps\in\R_+^\star$ such that
\begin{equation*}
\|(\uhnr)^-\|_{L^2(\Omega;L^2(0,T;L^2(D)))}^2+\|(\uhnr-1)^+\|_{L^2(\Omega;L^2(0,T;L^2(D)))}^2\leq M\eps^2,
\end{equation*}
and the announced result holds.
\end{proof}
\begin{rem}\label{cvtpsi2}
As explained in \cite[Section 3.2]{BBBLV}, since $-\frac{(\uhnr)^-}{\eps}\leq 0$ and $\frac{(\uhnr-1)^+}{\eps}\geq 0$, it follows that $\psi_1\leq 0$ and $\psi_2 \geq 0$. Using the fact that $\pe(\uhnr)=-\frac{(\uhnr)^-}{\eps}+\frac{(\uhnr-1)^{+}}{\eps}$, from Lemmas \ref{CVpeuhnr} and \ref{CVnpapp} we get that $\psi=\psi_1+\psi_2$. From
\begin{equation*}
\pe(\uhnr)\uhnr=\frac{(\uhnr)^-}{\eps}\times (\uhnr)^-+\frac{(\uhnr-1)^+}{\eps}\times \left((\uhnr-1)^++1\right),
\end{equation*}
one obtains thanks to Lemmas \ref{CVnpapp} and \ref{CVSppapn} that
\begin{equation*}
\lim_{m\to\infty}\erwb \int_0^T\int_D\pe(\uhnr(t,x))\uhnr(t,x)\,dx\,dt \erwe
= \erwb \int_0^T\int_{ D}\psi_2(t,x)\,dx\,dt \erwe.
\end{equation*}
\end{rem}
\begin{lem}\label{CVguhnl} There exists a process $\mathcal{G}$ in $L^2_{\mathcal{P}_T}(\Omega_T;\HS)$ such that, up to subsequences denoted in the same way, $\left(G(\uhnr)\right)_m$ and $(G(\uhnl))_m$ converge to $\mathcal{G}$ weakly in $L^2(\Omega;L^2(0,T;\HS)\text{ as }m\to\infty.$
\end{lem}
\begin{proof}
The existence result is a consequence of Lemma \ref{boundguhnlr_and_betauhnlr}, while the fact that the weak limit $\mathcal{G}$ is common follows from (\ref{G3_inequality_G_q2}) combined with the estimate \eqref{210824_05}.
The predictability property of $\mathcal{G}$ with values in $\HS$ is inherited from $\left(G(\uhnl)\right)_m$ at the limit.
\end{proof}
\begin{lem}\label{251002_l1}
The sequences $(M_{N,\eps})_m$ and $(\widehat{M}_{N,\eps})_m$ given by Definition \ref{piecewiseconstantbis} converge towards $\int_0^{\cdot}\mathcal{G}(s)\,dW(s)$ weakly in $L^2(\Omega;\mathscr{C}([0,T];L^2(D)))$ as $m\to\infty$.
\end{lem}
\begin{proof}
The weak convergence result for $(M_{N,\eps})_m$ follows from Lemma \ref{CVguhnl} and the fact that the stochastic It\^{o} integral operator
\begin{equation*}
\mathcal{I}: L^2_{\mathcal{P}_T}\big(\Omega_T;\HS\big)\to L^2(\Omega;\mathscr{C}([0,T];L^2(D))), \quad \mathscr{H}\mapsto \mathcal{I}(\mathscr{H})=\int_0^{\cdot}\mathscr{H}(s)\,dW(s)
\end{equation*}
is a linear mapping and, thanks to Burkholder-Davis-Gundy inequality (see, \textit{e.g.,} \cite[Theorem 4.36, p.114]{DPZ14}), also a (norm) continuous mapping. As a consequence, using the fact that $(M_{N,\eps})_{m}=(\mathcal{I}(G(\uhnl)))_{m}$, one deduces that the sequence $(M_{N,\eps})_{m}$ weakly converges towards $\mathcal{I}(\mathcal{G})$ in $L^2(\Omega;\mathscr{C}([0,T];L^2(D)))$ as $m\to\infty$.
Thanks to this weak convergence result, we can affirm that up to a subsequence denoted in the same way, the sequence $(M_{N,\eps})_m$ is bounded in $L^2(\Omega;\mathscr{C}([0,T];L^2(D)))$. Using the fact that $(M_{N,\eps}-\widehat{M}_{N,\eps})_m$ is bounded in $L^2(\Omega;\mathscr{C}([0,T];L^2(D)))$ by Lemma \ref{260409_lem1}, we then get the boundedness of $(\widehat{M}_{N,\eps})_m$ in the same space. Combining this with Lemma \ref{251007_l1}, the weak convergence of $(\widehat{M}_{N,\eps})_m$ towards the same limit as $(M_{N,\eps})_m$ follows immediately.
\end{proof}
\begin{lem}\label{251016_l2} Up to a subsequence denoted in the same way, $\left(\partial_t\left(\widehat{u}_{N,\eps}-\widehat{M}_{N,\eps}\right)\right)_m$ weakly converges towards $\partial_t\left(u-\int_0^\cdot \mathcal{G}(s)dW(s)\right)$ in $L^{p'}(\Omega;L^{p'}(0,T;V^*))$ as $m \to \infty$.
\end{lem}
\begin{proof}
Lemma \ref{251016_l1} gives us the boundedness of the sequence $\left(\partial_t\left(\widehat{u}_{N,\eps}-\widehat{M}_{N,\eps}\right)\right)_m$ in $L^{p'}(\Omega;L^{p'}(0,T;V^*))$. By combining this with the boundedness of $\left(\widehat{u}_{N,\eps}-\widehat{M}_{N,\eps}\right)_m$ in $L^2(\Omega_T;L^2(D))$ and its weak convergence towards $u-\int_0^\cdot \mathcal{G}(s)dW(s)$ in the same space (given by Lemmas \ref{addreg u} and \ref{251002_l1}), the announced result is immediate by continuity of the time derivative operator from $L^{p'}(\Omega;W^{1,(2,p')}(0,T;L^2(D), V^*))$ to $L^{p'}(\Omega;L^{p'}(0,T;V^*))$, where $W^{1,(2,p')}(0,T;L^2(D), V^*)=\{v\in L^2(0,T;L^2(D)): \partial_t v \in L^{p'}(0,T;V^*) \}$.
\end{proof}
\begin{lem}\label{CVbetauhnl} There exists a process $\mathcal{B}$ in $L^2_{\mathcal{P}_T}\big(\Omega_T;L^2(D)\big)$ such that, up to subsequences denoted in the same way, $(\beta(\uhnr))_m$ and $(\beta(\uhnl))_m$ converge weakly to $\mathcal{B}$ in $L^2(\Omega;L^2(0,T;L^2(D)))\text{ as }m\to\infty.$
\end{lem}
\begin{proof} The weak convergence is a consequence of Lemma \ref{boundguhnlr_and_betauhnlr}, while the joint limit $\mathcal{B}$ follows from the Lipschitz continuity of $\beta$ and \eqref{210824_05}.
\end{proof}
\subsection{The limit equation}
\begin{prop}\label{PTTL} The weak limit $u$ introduced in Proposition \ref{addreg u} is an element of $L^2(\Omega;\mathscr{C}([0,T];L^2(D)))$ and satisfies \Pas, for all $t\in [0,T]$
\begin{align}\label{251009_04}
\hspace*{-0.27cm}u(t)-u_0+\hspace*{-0.05cm}\int_0^t\hspace*{-0.05cm}\left(y(s)-\dvz(s)+\psi(s)\right)\,ds\hspace*{-0.05cm}=\hspace*{-0.05cm}\int_0^t \left(\mathcal{B}(s)+f(s)\right)\,ds+\hspace*{-0.05cm}\int_0^t \mathcal{G}(s)\,dW(s)
\end{align}
in $L^2(D)$, where $\vec{z}$, $y$, $\psi$, $\mathcal{G}$ and $\mathcal{B}$ are given by Lemmas \ref{251006_l1}, \ref{CVpeuhnr}, \ref{CVguhnl} and \ref{CVbetauhnl} respectively, and where $\diver^{\mathcal{N}}$ denotes the divergence operator associated with Neumann boundary conditions.
\end{prop}
\begin{proof}
Starting from the scheme \eqref{240911_01}, it is classical to prove that the sequences $(\partial_t(\widehat{u}_{N,\eps}-\widehat{M}_{N,\eps}))_m$, $(|\nabla\uhnr|^{p-2}\nabla\uhnr)_m$, $(|\uhnr|^{p-2}\uhnr)_m$, $(\pe(\uhnr))_m$, $(\beta(\uhnr))_m$ and $(f^N)_m$ satisfy the following equality for any $\varphi$ in $V$, any $\xi$ in $L^{p}(0,T)$ and any $A\in \mathcal{F}$
\begin{align*}
&\erww{\mathds{1}_A\int_0^T\int_{D}\partial_t(\widehat{u}_{N,\eps}-\widehat{M}_{N,\eps})(t,x)\varphi(x)\xi(t)\,dt\,dx}\\
&+\erww{\mathds{1}_A\int_0^T\int_{D}\left(|\nabla\uhnr(t,x)|^{p-2}\nabla\uhnr(t,x)\cdot\nabla \varphi(x)+|\uhnr(t,x)|^{p-2}\uhnr(t,x)\varphi(x)\right)\xi(t)\,dx\,dt }\\
&+\erww{\mathds{1}_A\int_0^T\int_{D}\pe(\uhnr(t,x))\varphi(x)\xi(t)\,dx\,dt}\\
=&\erww{\mathds{1}_A\int_0^T\int_{D}\left(\beta(\uhnr(t,x))+f^N(t,x)\right)\varphi(x)\xi(t)\,dx\,dt},
\end{align*}
which can be expressed as
\begin{equation}\label{250929_03}
I_{1,m}+I_{2,m}+I_{3,m} = I_{4,m},
\end{equation}
where
\begin{align*}
I_{1,m}=&\erww{\mathds{1}_A\int_0^T\int_{D}\partial_t(\widehat{u}_{N,\eps}-\widehat{M}_{N,\eps})(t,x)\varphi(x)\xi(t)\,dt\,dx}\\
I_{2,m}=&\erww{\mathds{1}_A\int_0^T\int_{D}\left(|\nabla\uhnr(t,x)|^{p-2}\nabla\uhnr(t,x)\cdot\nabla \varphi(x)+|\uhnr(t,x)|^{p-2}\uhnr(t,x)\varphi(x)\right)\xi(t)\,dx\,dt }\\
I_{3,m}=&\erww{\mathds{1}_A\int_0^T\int_{D}\pe(\uhnr(t,x))\varphi(x)\xi(t)\,dx\,dt}\\
I_{4,m}=&\erww{\mathds{1}_A\int_0^T\int_{D}\left(\beta(\uhnr(t,x))+f^N(t,x)\right)\varphi(x)\xi(t)\,dx\,dt}.
\end{align*}
From Lemma \ref{251016_l2}, using the fact that $\mathds{1}_{A}\varphi \xi\in L^p(\Omega;L^{p}(0,T;V))=\left(L^{p'}(\Omega;L^{p'}(0,T;V^*))\right)'$ one gets directly the following limit for $I_{1,m}$
\begin{equation*}
\lim_{m\to\infty} I_{1,m}=\erww{\mathds{1}_A\int_0^T\int_{D} \partial_t\left(u(x)-\int_0^\cdot \mathcal{G}(s,x)\,dW(s)\right)(t) \varphi(x)\xi(t)\,dx\,dt}.
\end{equation*}
Using Lemma \ref{251006_l1}, we obtain
\begin{align*}
\lim_{m\to\infty} I_{2,m} = \erww{\mathds{1}_A\int_0^T\int_D \big(\vec{z}(t,x)\cdot \nabla\varphi(x)+ y(t,x)\varphi(x)\big)\xi(t)\,dx\,dt}. 
\end{align*}
From Lemma \ref{CVpeuhnr}, we get
\begin{align*}
\lim_{m\to\infty} I_{3,m}=\erww{\mathds{1}_A\int_0^T \int_{ D} \psi(t,x) \varphi(x)\xi(t)\,dx\,dt}.
\end{align*}
Thanks to Lemmas \ref{CVbetauhnl} and \ref{CVfhnl} we arrive at
\begin{align*}
\lim_{m\to\infty}I_{4,m}=\erww{\mathds{1}_A\int_0^T \int_{D} \left(\mathcal{B}(t,x)+f(t,x)\right)\varphi(x)\xi(t)\,dx\,dt}.
\end{align*}
Passing to the limit in \eqref{250929_03} and using all the convergence results, we obtain
\begin{align}\label{170326_a}
&\erww{\mathds{1}_A\int_0^T\int_{D} \partial_t\left(u(x)-\int_0^\cdot \mathcal{G}(s,x)\,dW(s)\right)(t) \varphi(x)\xi(t)\,dx\,dt}\nonumber\\
&+\erww{\mathds{1}_A\int_0^T\int_D \big(\vec{z}(t,x)\cdot \nabla\varphi(x)+ y(t,x)\varphi(x)\big)\xi(t)\,dx\,dt}\nonumber\\
&+\erww{\mathds{1}_A\int_0^T \int_{ D} \psi(t,x) \varphi(x)\xi(t)\,dx\,dt}\nonumber\\
=&\erww{\mathds{1}_A\int_0^T \int_{D} \left(\mathcal{B}(t,x)+f(t,x)\right)\varphi(x)\xi(t)\,dx\,dt}.
\end{align}
Then, identifying the integral $\int_D v(x)\phi(x)dx$ with the duality bracket $\langle v,\phi\rangle_{V^*,V}$, for any $v$ in $L^{p'}(D)$ and any $\phi\in V$, and introducing the slightly hand-waving notation $\dvz$ defined for any $\phi$ in $V$ by
\begin{align}\label{HWN}
\langle \dvz,\phi\rangle_{V^*,V}
&\ := -\int_D \vec{z}(x) \cdot \nabla \phi(x)\,dx,
\end{align}
we may write
\begin{align}\label{170326_b}
&\erww{\mathds{1}_A\int_0^T \left\langle\partial_t\Big(u-\int_0^\cdot \mathcal{G}(s)\,dW(s)\Big)(t), \varphi\right\rangle_{V^*,V}\xi(t)\,dt}\nonumber\\
=&\erww{\mathds{1}_A\int_0^T\langle \dvz(t)-y(t),\varphi\rangle_{V^*,V}\xi(t)\,dt}-\erww{\mathds{1}_A\int_0^T \left\langle \psi(t), \varphi\right\rangle_{V^*,V}\xi(t)\,dt}\nonumber\\
&+\erww{\mathds{1}_A\int_0^T \left\langle\mathcal{B}(t)+f(t),\varphi\right\rangle_{V^*,V}\xi(t)\,dt}.
\end{align}
Incidentally, it is worth noting that using this notation, we can claim the weak convergence of $(\plap\uhnr)_m$ towards $-\dvz+y$ in $L^{p'}(\Omega_T;V^*)$ as $m\to\infty$.
Using a separability argument, one gets that $d\mathds{P}\otimes dt$-a.e. in $\Omega_T$, the following equality holds in $V^*$
\begin{align}\label{251009_05bis}
\partial_t\left(u-\int_0^\cdot\mathcal{G}(s)\,dW(s)\right)(t) = \dvz(t)-y(t)-\psi(t)+\mathcal{B}(t)+f(t).
\end{align}
Recalling that \Pas, since $2\geq p'$, $\big(u-\int_0^\cdot\mathcal{G}(s)\,dW(s)\big) \in L^{p'}(0,T;L^2(D))$ and $\partial_t \big(u-\int_0^\cdot\mathcal{G}(s)\,dW(s) \big)\in L^{p'}(0,T;V^*)$, then (see, \textit{e.g.,} \cite{Roubicek}), $\big(u-\int_0^\cdot\mathcal{G}(s)\,dW(s)\big)$ is an element of $\mathscr{C}([0,T];V^*)$. Applying \cite[Lemme 2.2.1 p.44]{D}, one gets that \Pas \ and for any $t\in [0,T]$, the following equality holds in $V^*$
\begin{equation}\label{180326_b}
\Big(u-\int_0^\cdot\mathcal{G}(s)\,dW(s)\Big)(t)-\Big(u-\int_0^\cdot\mathcal{G}(s)\,dW(s)\Big)(0)
= \int_0^t \partial_t \big(u-\int_0^\cdot\mathcal{G}(\sigma)\,dW(\sigma) \big)(s)\,ds.
\end{equation}
By combining \eqref{251009_05bis} with \eqref{180326_b}, and recalling that from the continuity of the stochastic integral it follows that \Pas,\ $u\in\mathscr{C}([0,T];V^*)$, then one gets that \Pas \ and for any $t\in [0,T]$, in $V^*$
\begin{equation}\label{31032026_a}
u(t)=u(0)+\int_0^t\left(\dvz(s)-y(s)-\psi(s)+\mathcal{B}(s)+f(s)\right)\,ds+\int_0^t \mathcal{G}(s)\,dW(s).
\end{equation}
In particular, defining $w:=u-\int_0^\cdot\mathcal{G}(s)\,dW(s)$, since $\partial_t w$ is the limit in $L^2_{\mathcal{P}_T}(\Omega_T, V^*)$ of the backward difference quotient $w_h:=\frac{w(\cdot)-w(\cdot-h)}{h}$ for $h \to 0$,
$\partial_t w$ is therefore predictable with values in $V^*$. Owing to
\eqref{251009_05bis}, it then follows that $-\dvz+y$ is predictable with values in $V^*$, 
and from the It\^{o} formula in \cite[Theorem 4.2.5]{LR} we get that \eqref{31032026_a} holds in $L^2(D)$ and that $u\in L^2(\Omega;\mathscr{C}([0,T];L^2(D)))$. \\
To complete the proof of Proposition \ref{PTTL}, it is left to show that $u(0)=u_0$ $d\mathds{P}\otimes dx$-a.e. in $\Omega\times D$.
To do so, we fix $t\in[0,T]$ and choose test functions $\xi\in\mathscr{D}([0,t])$ and $\varphi\in V$.
By the continuity of the linear mapping $v\in V^*\mapsto \langle v,\varphi\rangle_{V^*,V}\in \R$, it is true that for any
$g\in L^1(0,T;V^*)$, $\int_0^T \langle g(r),\varphi\rangle_{V^*,V}dr=\langle \int_0^T g(r)dr,\varphi\rangle_{V^*,V}$.
Using this permutation result combined with integration by parts, one gets that \Pas,
\begin{align}\label{251010_01a}
&\int_0^{t}\left\langle\partial_r\left(\widehat{u}_{N,\eps}-\widehat{M}_{N,\eps}\right)(r)\xi(r),\varphi\right\rangle_{V^*,V}\,dr\nonumber\\
=\ &\left\langle \int_0^t\partial_r\left(\widehat{u}_{N,\eps}-\widehat{M}_{N,\eps}\right)(r)\xi(r)\,dr,\varphi\right\rangle_{V^*,V}\nonumber\\
=\ &-\int_0^t\int_D(\widehat{u}_{N,\eps}-\widehat{M}_{N,\eps})(r,x)\xi'(r)\varphi(x)\,dx\,dr\nonumber\\
&+\int_D\left((\widehat{u}_{N,\eps}-\widehat{M}_{N,\eps})(t,x)\xi(t)-u_0(x)\xi(0)\right)\varphi(x)\,dx.
\end{align}
We recall that, according to Lemma \ref{251016_l2}, the sequence $\big(\partial_t\big(\widehat{u}_{N,\eps}-\widehat{M}_{N,\eps}\big)\big)_m$ weakly converges towards $\partial_t\left(u-\int_0^{\cdot}\mathcal{G}(s)\,dW(s)\right)$ in $L^{p'}(\Omega;L^{p'}(0,T;V^*))$ and by Lemma \ref{251016_l3}, $(\widehat{u}_{N,\eps}(t))_{m_t}$ weakly converges towards $\chi_t$ in $L^2(\Omega\times D)$.
Moreover, owing to Lemma \ref{251002_l1}, $(\widehat{M}_{N,\eps})_m$ weakly converges towards $\int_0^{\cdot}\mathcal{G}(s)\,dW(s)$ in $L^2(\Omega;\mathscr{C}([0,T];L^2(D)))$. Since the operator
\begin{equation*}
\delta_t:L^2(\Omega;\mathscr{C}([0,T];L^2(D)))\to L^2(\Omega\times D), \quad v\mapsto v(t,\cdot)
\end{equation*}
is continuous and linear, it follows that $(\widehat{M}_{N,\eps}(t))_{m_t}$ weakly converges towards
$\int_0^t\mathcal{G}(s)\,dW(s)$ in $L^2(\Omega\times D)$.
On one hand, multiplying \eqref{251010_01a} with $\mathds{1}_A$ for $A\in\mathcal{F}$, taking expectation, passing to the limit with $m_t\to \infty$, we obtain
\begin{align}\label{251010_02a}
\begin{aligned}
&\erww{\mathds{1}_A\int_0^{t}\left\langle\partial_t\left(u-\int_0^\cdot\mathcal{G}(s)\,dW(s)\right)(r) \xi(r),\varphi\right\rangle_{V^*,V}\,dr}\\
=\ &-\erww{\mathds{1}_A\int_0^t\int_D \left(u(r,x)-\int_0^r\mathcal{G}(s,x)\,dW(s)\right)\xi'(r)\varphi(x)\,dx\,dr}\\
&+\erww{\mathds{1}_A\int_D\left(\left(\chi_t(x)-\int_0^t\mathcal{G}(s,x)\,dW(s)\right)\xi(t)-u_0(x)\xi(0)\right)\varphi(x)\,dx}.
\end{aligned}
\end{align}
On the other hand, since $u\in L^2(\Omega;\mathscr{C}([0,T];L^2(D)))$, using integration by parts on the left-hand side of \eqref{251010_02a}
and after rearranging the terms that do not cancel, we arrive at
\begin{equation}\label{251010_03a}
\erww{\mathds{1}_A\int_D \big(u(t,x)\xi(t)-u(0,x)\xi(0)\big)\varphi(x)\,dx}=\erww{\mathds{1}_A\int_D \left(\chi_t(x)\xi(t)-u_0(x)\xi(0)\right)\varphi(x)\,dx}
\end{equation}
In particular, by choosing $\xi\in\mathscr{D}([0,t])$ such that $\xi(t)=0$ and $\xi(0)=1$ in \eqref{251010_03a}, we find that $u(0)=u_0,$ $d\mathds{P}\otimes dx$-a.e. in $\Omega\times D$.
\end{proof}
\begin{lem}
For any fixed $t\in [0,T]$, the sequence $(\widehat{u}_{N,\eps}(t))_{m}$ converges weakly towards $u(t)$ in $L^2(\Omega\times D)$ as $m\to\infty$.
\end{lem}
\begin{proof} We fix $t\in[0,T]$, $A\in \mathcal{F}$ and choose in \eqref{251010_03a} test functions $\varphi\in V$ and $\xi\in\mathscr{D}([0,t])$ such that $\xi(t)=1$ and $\xi(0)=0$ which leads us to
\begin{equation*}
\erww{\mathds{1}_A\int_D u(t,x)\varphi(x)\,dx}=\erww{\mathds{1}_A\int_D \chi_t(x)\varphi(x)\,dx}.
\end{equation*}
As a consequence, we find that for any $t\in [0,T]$, $\chi_t=u(t)$ \Pas \ and a.e in $D$, and that the whole sequence $(\widehat{u}_{N,\eps}(t))_m$ converges weakly to $u(t)$ in $L^2(\Omega\times D)$ as $m\to\infty$.
\end{proof}
\subsection{Identification of weak limits coming from non-linear terms}
Now, we have all the necessary tools for the identification of $\psi$, $\psi_1$, $\psi_2$, $-\diver\vec{z}+y$, $\mathcal{G}$ and $\mathcal{B}$, and for completing the proof of Theorem \ref{main}. This identification procedure will be possible by combining monotonicity properties of the $p$-Laplace operator with computations developed in \cite[Proposition 5.12]{BSVZ2025} and adapted to our case.
For this reason, let us mention that the beginning of the proof of Proposition \ref{PISL} below will be similar to the one of \cite[Proposition 5.12]{BSVZ2025} but for the sake of clarity and completeness, most of the computations will be detailed here.
\begin{prop}\label{PISL}
The pair $(u, \psi)$, where $u$ and $\psi$ are introduced in Lemmas \ref{addreg u} and \ref{CVpeuhnr}, respectively, is a solution of Problem \eqref{equation} in the sense of Definition \ref{solution}.
\end{prop}
\begin{proof}
Set $k\in\{0,\ldots,N-1\}$, $\eps\in\R_+^\star$ and an arbitrary constant $c>0$, multiply the scheme \eqref{240911_01} with $e^{-c t_k} \upskp$, use the identity $a(a-b)=\frac{1}{2}(a^2-b^2+(a-b)^2)$, integrate over $D$ and take expectation to arrive at
\begin{align}\label{310326_a}
&\frac{1}{2} e^{- c t_k} \left( \erwb \|\upskp\|_{L^2(D)}^2 \erwe - \erwb \|\upsk\|_{L^2(D)}^2 \erwe + \erwb \|\upskp-\upsk\|_{L^2(D)}^2 \erwe \right)\nonumber\\
&+ e^{- c t_k} \tau \erwb \langle A_\eps(\upskp), \upskp \rangle_{V^*,V} \erwe\nonumber\\
=\ & e^{- c t_k} \erwb (G(\upsk) \Delta_{k+1} W, \upskp)_{L^2(D)} \erwe + e^{- c t_k} \tau \erwb (\beta(\upskp) + f_k, \upskp)_{L^2(D)} \erwe,
\end{align}
where $A_\eps = \plap + \psi_\eps.$
By applying Young's inequality and \cite[Proposition 4.20]{DPZ14} as for obtaining the estimate \eqref{term4}, we get
\begin{equation}\label{2}
\erwb (G(\upsk)\Delta_{k+1}W,\upskp)_{L^2(D)} \erwe \leq \frac{\tau}{2}\erwb \| G(\upsk) \|_{\HS}^2 \erwe + \frac{1}{2}\erwb \| \upskp-\upsk\|_{L^2(D)}^2 \erwe.
\end{equation}
Plugging \eqref{2} into \eqref{310326_a} and summing over $k\in\{0,\ldots,n\}$ for any $n\in\{0,\ldots,N-1\}$ yields
\begin{align}\label{4}
&\frac{1}{2} \sum_{k=0}^n e^{- c t_k} \left( \erwb \|\upskp\|_{L^2(D)}^2 \erwe - \erwb \|\upsk\|_{L^2(D)}^2 \erwe \right)
+\sum_{k=0}^n e^{- c t_k} \tau \erwb \langle A_\eps(\upskp), \upskp \rangle_{V^*,V} \erwe \nonumber\\
\leq\ & \sum_{k=0}^n e^{- c t_k} \frac{\tau}{2}\erwb \| G(\upsk) \|_{\HS}^2 \erwe
+ \sum_{k=0}^n e^{- c t_k} \tau \erwb (\beta(\upskp) + f_k, \upskp)_{L^2(D)} \erwe.
\end{align}
Let us study each term of \eqref{4} separately. \\
$\bullet$ Firstly, using the notation $t_{-1} = -\tau$, we can write that for any $k\in \{0,\ldots,n\}$
\begin{align*}
e^{- c t_k} \erwb \| \upskp \|_{L^2(D)}^2 - \| \upsk \|_{L^2(D)}^2 \erwe
=\ & e^{- c t_k} \erwb \| \upskp \|_{L^2(D)}^2 \erwe
- e^{- c t_{k-1}} \erwb \| \upsk \|_{L^2(D)}^2 \erwe\\
&- (e^{- c t_k} - e^{- c t_{k-1}}) \erwb \| \upsk \|_{L^2(D)}^2 \erwe.
\end{align*}
On one hand, we have
\begin{align*}
&\sum_{k=0}^n \left( e^{- c t_k} \erwb \| \upskp \|_{L^2(D)}^2 \erwe
- e^{- c t_{k-1}} \erwb \| \upsk \|_{L^2(D)}^2 \erwe \right)\\
=\ & e^{- c t_n} \erwb \| \upsnp \|_{L^2(D)}^2 \erwe
- e^{c \tau} \erwb \| u_0 \|_{L^2(D)}^2 \erwe.
\end{align*}
On the other hand, using the equality $e^{- c t_{k+1}} - e^{- c t_k} = \int_{t_k}^{t_{k+1}} (- c) e^{- c s}\,ds$, we get that
\begin{align*}
&\sum_{k=0}^n (e^{- c t_k} - e^{- c t_{k-1}}) \erwb \| \upsk \|_{L^2(D)}^2 \erwe\\
=\ & (1 - e^{c \tau}) \erwb \| u_0 \|_{L^2(D)}^2 \erwe
+ \sum_{k=0}^{n-1} \int_{t_k}^{t_{k+1}} (- c) e^{- c s}\,ds \erwb \| \upskp \|_{L^2(D)}^2 \erwe\\
=\ & (1 - e^{c \tau}) \erwb \| u_0 \|_{L^2(D)}^2 \erwe
- c \int_{0}^{t_n} e^{- c s} \erwb \| \uhnr(s) \|_{L^2(D)}^2 \erwe\,ds.
\end{align*}
Hence, with the inequality $e^{c \tau} \geq 1$,
\begin{align}\label{5}
&\frac{1}{2} \sum_{k=0}^n e^{- c t_k} \erwb \| \upskp \|_{L^2(D)}^2 - \| \upsk \|_{L^2(D)}^2 \erwe\\
\geq \ & \frac{1}{2} \left( e^{- c t_n} \erwb \| \upsnp \|_{L^2(D)}^2 \erwe
- e^{c \tau} \erwb \| u_0 \|_{L^2(D)}^2 \erwe \right)
+ \frac{c}{2} \int_{0}^{t_k} e^{- c s} \erwb \| \uhnr(s) \|_{L^2(D)}^2 \erwe\,ds.\nonumber
\end{align}
$\bullet$ Secondly, since $\pe$ is monotone we have
\begin{align*}
&\erwb\langle A_\eps(\upskp), \upskp \rangle_{V^*,V}\erwe = \erwb\int_D \left( |\nabla \upskp(x)|^p + |\upskp(x)|^p + \pe(\upskp(x)) \upskp(x) \right)\,dx\erwe \geq 0,
\end{align*}
and note that for all $s\in[t_k,t_{k+1}]$, $e^{- c t_k} \geq e^{- c s}$ so
\begin{align*}
&\tau e^{- c t_k} \erwb \langle A_\eps(\upskp), \upskp \rangle_{V^*,V} \erwe \geq \int_{t_k}^{t_{k+1}} e^{- c s} \erwb \langle A_\eps(\upskp), \upskp \rangle_{V^*,V} \erwe\,ds,
\end{align*}
which leads to
\begin{align}\label{3}
&\sum_{k=0}^n \tau e^{- c t_k} \erwb \langle A_\eps(\upskp), \upskp \rangle_{V^*,V} \erwe \geq \int_{0}^{t_{n+1}} e^{- c s} \erwb \langle A_\eps(\upskp), \upskp \rangle_{V^*,V} \erwe\,ds.
\end{align}
$\bullet$ Thirdly, we have
\begin{align*}
\sum_{k=0}^n e^{- c t_k} \frac{\tau}{2} \erwb \| G(\upsk) \|_{\HS}^2 \erwe
&= \frac{\tau}{2} \erwb \| G(u_0) \|_{\HS}^2 \erwe
+ \frac{\tau}{2} \sum_{k=0}^{n-1} e^{- c t_{k+1}} \erwb \| G(\upskp) \|_{\HS}^2 \erwe,
\end{align*}
and since for all $s\in[t_k,t_{k+1}], e^{- c t_{k+1}} \leq e^{- c s}$, we obtain
\begin{align}\label{6}
&\sum_{k=0}^n e^{- c t_k} \frac{\tau}{2} \erwb \| G(\upsk) \|_{\HS}^2 \erwe\nonumber\\
\leq\ & \frac{\tau}{2} \erwb \| G(u_0) \|_{\HS}^2 \erwe
+ \frac{1}{2} \int_0^{t_n} e^{- c s} \erwb \| G(\uhnr(s)) \|_{\HS}^2 \erwe\,ds.
\end{align}
$\bullet$ Fourthly, following \cite[Proposition 5.12]{BSVZ2025}, one gets the following estimate for the last term
\begin{align}\label{7}
&\sum_{k=0}^n e^{- c t_k} \tau \erwb ( \beta(\upskp) + f_k, \upskp )_{L^2(D)} \erwe\nonumber\\
\leq\ & \int_0^{t_{n+1}} e^{- c s} \erwb ( \beta(\uhnr(s)) + f^N(s), \uhnr(s) )_{L^2(D)} \erwe\,ds \\
&+ c \tau\left( L_\beta \| \uhnr \|_{L^2(\Omega;L^2(0,T;L^2(D)))}^2 + \| f \|_{L^2(\Omega;L^2(0,T;L^2(D)))} \| \uhnr\|_{L^2(\Omega;L^2(0,T;L^2(D)))} \right).\nonumber
\end{align}
Inserting \eqref{5}, \eqref{3}, \eqref{6} and \eqref{7} into \eqref{4}, and multiplying by $2$, we obtain for all $n\in\{0,\ldots,N-1\}$,
\begin{align*}
&e^{- c t_n} \erwb \| \upsnp \|_{L^2(D)}^2 \erwe
- e^{c \tau} \erwb \| u_0 \|_{L^2(D)}^2 \erwe
+ c \int_{0}^{t_n} e^{- c s} \erwb \| \uhnr(s) \|_{L^2(D)}^2 \erwe\,ds\\
&+ 2 \int_0^{t_{n+1}} e^{- c s} \erwb \langle A_\eps(\upskp), \upskp \rangle_{V^*,V} \erwe\,ds\\
\leq\ &
\tau \erwb \| G(u_0) \|_{\HS}^2 \erwe
+ \int_0^{t_n} e^{- c s} \erwb \| G(\uhnr(s)) \|_{\HS}^2 \erwe\,ds\\
&+ 2 \int_0^{t_{n+1}} e^{- c s} \erwb ( \beta(\uhnr(s)) + f^N(s), \uhnr(s) )_{L^2(D)} \erwe\,ds\\
&+2 c \tau\left( L_\beta\| \uhnr \|_{L^2(\Omega;L^2(0,T;L^2(D)))}^2 + \| f \|_{L^2(\Omega;L^2(0,T;L^2(D)))} \| \uhnr\|_{L^2(\Omega;L^2(0,T;L^2(D)))} \right).
\end{align*}
For $t\in[t_n,t_{n+1})$, since $e^{- c t}\leq e^{- c t_n}$ and $(t-\tau)^{+}\leq t_n$, one gets thanks to the non-negativity of the concerned integrands
\begin{align*}
&e^{- c t} \erwb \| \uhnr(t) \|_{L^2(D)}^2 \erwe
- e^{c \tau} \erwb \| u_0 \|_{L^2(D)}^2 \erwe + c \int_{0}^{(t-\tau)^{+}} e^{- c s} \erwb \| \uhnr(s) \|_{L^2(D)}^2 \erwe\,ds\\
&+ 2 \int_0^t e^{- c s} \erwb \langle A_\eps(\uhnr(s)), \uhnr(s) \rangle_{V^*,V} \erwe\,ds\nonumber\\
\leq\ &
\tau \erwb \| G(u_0) \|_{\HS}^2 \erwe
+ \int_0^{t} e^{- c s} \erwb \| G(\uhnr(s)) \|_{\HS}^2 \erwe\,ds\\
&+ 2 \int_0^{t} e^{- c s} \erwb \big( \beta(\uhnr(s)) + f^N(s), \uhnr(s) \big)_{L^2(D)} \erwe\,ds\\
&+ 2 \int_t^{t_{n+1}} e^{- c s} \erwb \big( \beta(\uhnr(s)) + f^N(s), \uhnr(s) \big)_{L^2(D)} \erwe\,ds\\
&+2 c \tau\left( L_\beta \| \uhnr \|_{L^2(\Omega;L^2(0,T;L^2(D)))}^2 + \| f \|_{L^2(\Omega;L^2(0,T;L^2(D)))} \| \uhnr\|_{L^2(\Omega;L^2(0,T;L^2(D)))} \right).
\end{align*}
Thanks to the constant $K_0>0$ introduced in Proposition \ref{bounds}, we have
\begin{align*}
&c \int_{(t-\tau)^{+}}^t e^{- c s} \erwb \| \uhnr(s) \|_{L^2(D)}^2 \erwe\,ds
+2 \int_t^{t_{n+1}} e^{- c s} \erwb \big( \beta(\uhnr(s)) + f^N(s), \uhnr(s) \big)_{L^2(D)} \erwe\,ds\\
\leq \ & \tau(c+2L_\beta)K_0 + 2\sqrt{\tau}\| f \|_{L^2(\Omega;L^2(0,T;L^2(D)))}\sqrt{K_0} ,
\end{align*}
and using the fact that $-\int_{0}^{(t-\tau)^{+}}=-\int_0^t+\int_{(t-\tau)^{+}}^t$, one gets 
\begin{align}\label{020426_a}
&e^{- c t} \erwb \| \uhnr(t) \|_{L^2(D)}^2 \erwe
- e^{c \tau} \erwb \| u_0 \|_{L^2(D)}^2 \erwe + 2 \int_0^t e^{- c s} \erwb \langle A_\eps(\uhnr(s)), \uhnr(s) \rangle_{V^*,V} \erwe\,ds\nonumber\\
\leq\ &
\tau \erwb \| G(u_0) \|_{\HS}^2 \erwe
+ \int_0^{t} e^{- c s} \erwb \| G(\uhnr(s)) \|_{\HS}^2 \erwe\,ds\nonumber\\
&+ 2 \int_0^{t} e^{- c s} \erwb \big( \beta(\uhnr(s)) + f^N(s), \uhnr(s) \big)_{L^2(D)} \erwe\,ds\nonumber\\
&- c \int_{0}^{t} e^{- c s} \erwb \| \uhnr(s) \|_{L^2(D)}^2 \erwe\,ds\\
&+\tau(c+2L_\beta)K_0 + 2\sqrt{\tau}\| f \|_{L^2(\Omega;L^2(0,T;L^2(D)))}\sqrt{K_0}\nonumber\\
&+2 c \tau\left( L_\beta\| \uhnr \|_{L^2(\Omega;L^2(0,T;L^2(D)))}^2 + \| f \|_{L^2(\Omega;L^2(0,T;L^2(D)))}\| \uhnr\|_{L^2(\Omega;L^2(0,T;L^2(D)))} \right).\nonumber
\end{align}
Following \cite{BSVZ2025}, we consider the three following expansions for the right hand side terms of \eqref{020426_a} depending on $\uhnr$:
\begin{align*}
&\int_0^t e^{- c s} \erwb \| G(\uhnr(s)) \|_{\HS}^2 \erwe\,ds\\
=\ &
\int_0^t e^{- c s} \erwb \| G(\uhnr(s)) - G(u(s)) \|_{\HS}^2 \erwe\,ds\\
&+ 2 \int_0^t e^{- c s} \erwb (G(\uhnr(s)), G(u(s)))_{\HS} \erwe\,ds
- \int_0^t e^{- c s} \erwb \| G(u(s)) \|_{\HS}^2 \erwe\,ds,
\end{align*}
and
\begin{align*}
&2 \int_0^t e^{- c s} \erwb (\beta(\uhnr(s)), \uhnr(s))_{L^2(D)} \erwe\,ds\\
=\ &
2 \int_0^t e^{- c s} \erwb (\beta(\uhnr(s)) - \beta(u(s)), \uhnr(s) - u(s))_{L^2(D)} \erwe\,ds\\
&+ 2 \int_0^t e^{- c s} \erwb (\beta(\uhnr(s)), u(s))_{L^2(D)} \erwe\,ds
+ 2 \int_0^t e^{- c s} \erwb (\beta(u(s)), \uhnr(s) - u(s))_{L^2(D)} \erwe\,ds,
\end{align*}
and
\begin{align*}
&- c \int_0^t e^{- c s} \erwb \| \uhnr(s) \|_{L^2(D)}^2 \erwe\,ds\\
=\ &
- c \int_0^t e^{- c s} \erwb \| \uhnr(s) - u(s) \|_{L^2(D)}^2 \erwe\,ds\\
&- 2 c \int_0^t e^{- c s} \erwb (\uhnr(s), u(s))_{L^2(D)} \erwe\,ds
+ c \int_0^t e^{- c s} \erwb \| u(s) \|_{L^2(D)}^2 \erwe\,ds.
\end{align*}
For the rest of the proof, we choose a constant $c\geq 0$ with $ c> L_{G,2}^2 + 2 L_\beta$. In this way, we obtain
\begin{align*}
&\int_0^t e^{- c s} \erwb \| G(\uhnr(s)) - G(u(s)) \|_{\HS}^2 \erwe\,ds
- c \int_0^t e^{- c s} \erwb \| \uhnr(s) - u(s) \|_{L^2(D)}^2 \erwe\,ds\\
&+ 2 \int_0^t e^{- c s} \erwb (\beta(\uhnr(s)) - \beta(u(s)), \uhnr(s) - u(s))_{L^2(D)} \erwe\,ds
\leq 0.
\end{align*}
Using all of this and going back to \eqref{020426_a}, one gets for any $t\in[t_n,t_{n+1})$
\begin{align*}
&e^{- c t} \erwb \| \uhnr(t) \|_{L^2(D)}^2 \erwe 
+ 2 \int_0^t e^{- c s} \erwb \langle A_\eps(\uhnr(s)), \uhnr(s) \rangle_{V^*,V} \erwe\,ds\nonumber\\
\leq\ & 2 \int_0^t e^{- c s} \erwb (G(\uhnr(s)), G(u(s)))_{\HS} \erwe\,ds
- \int_0^t e^{- c s} \erwb \| G(u(s)) \|_{\HS}^2 \erwe\,ds
\nonumber\\
&+2\int_0^t e^{- c s} \erwb (\beta(\uhnr(s)), u(s))_{L^2(D)} \erwe\,ds
+ 2\int_0^t e^{- c s} \erwb (\beta(u(s)), \uhnr(s) - u(s))_{L^2(D)} \erwe\,ds,\nonumber\\
&- 2 c \int_0^t e^{- c s} \erwb (\uhnr(s), u(s))_{L^2(D)} \erwe\,ds
+ c \int_0^t e^{- c s} \erwb \| u(s) \|_{L^2(D)}^2 \erwe\,ds\nonumber\\
&+ 2 \int_0^{t} e^{- c s} \erwb \big(f^N(s), \uhnr(s) \big)_{L^2(D)} \erwe\,ds
+ \erwb \| u_0 \|_{L^2(D)}^2 \erwe
+ (e^{c \tau}-1) \erwb \| u_0 \|_{L^2(D)}^2 \erwe \nonumber\\
&+ \tau \erwb \| G(u_0) \|_{\HS}^2 \erwe
+ \tau(c+2L_\beta)K_0+2\sqrt{\tau}\| f \|_{L^2(\Omega;L^2(0,T;L^2(D)))}\sqrt{K_0}\nonumber\\
&+2 c \tau\left( L_\beta\| \uhnr \|_{L^2(\Omega;L^2(0,T;L^2(D)))}^2 + \| f \|_{L^2(\Omega;L^2(0,T;L^2(D)))}\| \uhnr\|_{L^2(\Omega;L^2(0,T;L^2(D)))} \right).\nonumber
\end{align*}
Then, integrating between $t_n$ and $t_{n+1}$ and summing over $k\in\{0,\ldots,N-1\}$ leads us to
\begin{align*}
&\int_0^Te^{- c t} \erwb \| \uhnr(t) \|_{L^2(D)}^2 \erwe\,dt
+ 2 \int_0^T\int_0^t e^{- c s} \erwb \langle A_\eps(\uhnr(s)), \uhnr(s) \rangle_{V^*,V} \erwe\,ds\,dt\nonumber\\
\leq\ & 2 \int_0^T\int_0^t e^{- c s} \erwb (G(\uhnr(s)), G(u(s)))_{\HS} \erwe\,ds\,dt\nonumber\\
&- \int_0^T\int_0^t e^{- c s} \erwb \| G(u(s)) \|_{\HS}^2 \erwe\,ds\,dt,
\nonumber\\
&+2\int_0^T\int_0^t e^{- c s} \erwb (\beta(\uhnr(s)), u(s))_{L^2(D)} \erwe\,ds\,dt\nonumber\\
&+ 2\int_0^T\int_0^t e^{- c s} \erwb (\beta(u(s)), \uhnr(s) - u(s))_{L^2(D)} \erwe\,ds\,dt\nonumber\\
&- 2 c \int_0^T\int_0^t e^{- c s} \erwb (\uhnr(s), u(s))_{L^2(D)} \erwe\,ds\,dt
+ c \int_0^T\int_0^t e^{- c s} \erwb \| u(s) \|_{L^2(D)}^2 \erwe\,ds\,dt\nonumber\\
&+ 2 \int_0^T\int_0^{t} e^{- c s} \erwb \big(f^N(s), \uhnr(s) \big)_{L^2(D)} \erwe\,ds\,dt\nonumber\\
&+ \int_0^T\erwb \| u_0 \|_{L^2(D)}^2 \erwe\,dt
+ (e^{c \tau}-1) \erwb \| u_0 \|_{L^2(D)}^2 T \erwe \nonumber\\
&+ \tau \erwb \| G(u_0) \|_{\HS}^2 \erwe T
+ \tau(c+2L_\beta)K_0T +2\sqrt{\tau}\| f \|_{L^2(\Omega;L^2(0,T;L^2(D)))}\sqrt{K_0}T\nonumber\\
&+ 2 c \tau\left( L_\beta \| \uhnr \|_{L^2(\Omega;L^2(0,T;L^2(D)))}^2 + \| f \|_{L^2(\Omega;L^2(0,T;L^2(D)))}\| \uhnr\|_{L^2(\Omega;L^2(0,T;L^2(D)))} \right)T.\nonumber
\end{align*}
From the inequality above, we pass to the superior limit and obtain owing to Remark \ref{cvtpsi2} and Lemma \ref{CVfhnl}
\begin{align}\label{020426_b}
&\liminf_{m\to\infty}\int_0^Te^{- c t} \erwb \| \uhnr(t) \|_{L^2(D)}^2 \erwe\,dt +2\int_0^T\int_0^t e^{- c s}\erwb \int_D \psi_2(s,x)\,dx\erwe\,ds\,dt\nonumber\\
&+ 2\limsup_{m\to\infty} \int_0^T\int_0^t e^{- c s} \erwb \langle \plap(\uhnr(s)), \uhnr(s) \rangle_{V^*,V} \erwe\,ds\,dt\nonumber\\
\leq\ & 2 \int_0^T\int_0^t e^{- c s} \erwb (\mathcal{G}(s), G(u(s)))_{\HS} \erwe\,ds\,dt
- \int_0^T\int_0^t e^{- c s} \erwb \| G(u(s)) \|_{\HS}^2 \erwe\,ds\,dt,
\nonumber\\
&+2\int_0^T\int_0^t e^{- c s} \erwb (\mathcal{B}(s), u(s))_{L^2(D)} \erwe\,ds\,dt- c \int_0^T\int_0^t e^{- c s} \erwb \| u(s) \|_{L^2(D)}^2 \erwe\,ds\,dt\nonumber\\
&+ 2 \int_0^T\int_0^{t} e^{- c s} \erwb \big(f(s), u(s) \big)_{L^2(D)} \erwe\,ds\,dt+ \int_0^T\erwb \| u_0 \|_{L^2(D)}^2 \erwe\,dt.
\end{align}
Moreover, by applying It\^o's formula \cite[Theorem 4.2.5]{LR} for Hilbert-Schmidt processes to the functional $(t,v) \mapsto e^{- c t} \|v\|_{L^2(D)}^2$ defined on $[0,T]\times L^2(D)$ and to the stochastic process $u$, for any $t\in [0,T]$, we arrive at
\begin{align}\label{020426_c}
&e^{- c t}\erwb \| u(t) \|_{L^2(D)}^2\erwe + 2\erwb \int_0^t e^{- c s} \langle -\dvz(s)+y(s), u(s) \rangle_{V^*,V}\,ds\erwe\nonumber\\
=\ & \erwb\| u_0 \|_{L^2(D)}^2\erwe - c\erwb \int_0^t e^{- c s} \| u(s) \|_{L^2(D)}^2\,ds\erwe\\
&+ 2 \erwb\int_0^t e^{- c s} (\mathcal{B}(s)+f(s)-\psi(s), u(s))_{L^2(D)}\,ds\erwe
+ \erwb\int_0^t e^{- c s} \| \mathcal{G}(s) \|_{\HS}^2\,ds\erwe.\nonumber
\end{align}
By plugging \eqref{020426_c} into \eqref{020426_b}, we get
\begin{align*}
&\liminf_{m\to\infty}\int_0^Te^{- c t} \erwb \| \uhnr(t) \|_{L^2(D)}^2 \erwe\,dt +2\int_0^T\int_0^t e^{- c s}\erwb \int_D \psi_2(s,x)\,dx\erwe\,ds\,dt\nonumber\\
&+ 2\limsup_{m\to\infty} \int_0^T\int_0^t e^{- c s} \erwb \langle \plap(\uhnr(s)), \uhnr(s) \rangle_{V^*,V} \erwe\,ds\,dt\nonumber\\
\leq\ & 2 \int_0^T\int_0^t e^{- c s} \erwb (\mathcal{G}(s), G(u(s)))_{\HS} \erwe\,ds\,dt\nonumber\\
&- \int_0^T\int_0^t e^{- c s} \erwb \| G(u(s)) \|_{\HS}^2 \erwe\,ds\,dt,
\nonumber\\
&+\int_0^Te^{- c t}\erwb \| u(t) \|_{L^2(D)}^2\erwe\,dt + 2\int_0^T\erwb \int_0^t e^{- c s} \langle -\dvz(s)+y(s), u(s) \rangle_{V^*,V}\,ds\erwe\,dt\nonumber\\
&+ 2\int_0^T \erwb\int_0^t e^{- c s} (\psi(s), u(s))_{L^2(D)}\,ds\erwe\,dt-\int_0^T\erwb\int_0^t e^{- c s} \| \mathcal{G}(s) \|_{\HS}^2\,ds\erwe\,dt,
\end{align*}
which can be rewritten as
\begin{align}\label{020426_d}
&\liminf_{m\to\infty}\int_0^Te^{- c t} \erwb \| \uhnr(t) \|_{L^2(D)}^2 \erwe\,dt\nonumber\\
&+2\int_0^T\int_0^t e^{- c s}\erwb \int_D \big(\psi_2(s,x)-\psi(s,x)u(s,x)\big)\,dx\erwe\,ds\,dt\nonumber\\
&+ \int_0^T\int_0^t e^{- c s} \erwb \| G(u(s))-\mathcal{G}(s) \|_{\HS}^2 \erwe\,ds\,dt,\nonumber\\
&+ 2\limsup_{m\to\infty} \int_0^T\int_0^t e^{- c s} \erwb \langle \plap(\uhnr(s)), \uhnr(s) \rangle_{V^*,V} \erwe\,ds\,dt\\
\leq\ &\int_0^Te^{- c t}\erwb \| u(t) \|_{L^2(D)}^2\erwe\,dt + 2\int_0^T \int_0^t e^{- c s} \erwb\langle -\dvz(s)+y(s), u(s) \rangle_{V^*,V}\,ds\erwe\,dt.\nonumber
\end{align}
From \cite[Proposition 5.12]{BSVZ2025}, Lemma \ref{CVSppapn} and Remark \ref{cvtpsi2}, it is possible to prove that $\psi_2-\psi u\geq 0$, \Pas \ and almost everywhere in $(0,T)\times D$. Using the weak convergence of $(\uhnr)_m$ towards $u$ in $L^2(\Omega;L^2(0,T;L^2(D)))$, the following inequality holds true
\begin{equation*}
\int_0^Te^{- c t}\erwb \| u(t) \|_{L^2(D)}^2\erwe\,dt\leq \liminf_{m\to\infty}\int_0^Te^{- c t} \erwb \| \uhnr(t) \|_{L^2(D)}^2 \erwe\,dt
\end{equation*}
and these pieces of information refine inequality \eqref{020426_d} into
\begin{align}\label{020426_e}
& \limsup_{m\to\infty} \int_0^T\int_0^t e^{- c s} \erwb \langle \plap(\uhnr(s)), \uhnr(s) \rangle_{V^*,V} \erwe\,ds\,dt\nonumber\\
\leq\ &\int_0^T \int_0^t e^{- c s} \erwb\langle -\dvz(s)+y(s), u(s) \rangle_{V^*,V} \erwe\,ds\,dt.
\end{align}
Using an adaptation of the famous Minty's trick for maximal monotone operators to the stochastic case (see \cite[Theorem 2.1]{Lions} or \cite{barbu10}), as proposed by \cite{PrevotRockner}, and an additional integration with respect to time, for our $p$-Laplace operator $\plap$ (see \cite{PrevotRockner} Examples 4.1.5 p.60 and 4.1.9 p.65 combined with the proof of Theorem 4.2.4 p.86), we can conclude directly from \eqref{020426_e} that almost everywhere in $\Omega_T$, $-\dvz+y=\plap u$ in $V^*$, and, by classical arguments, that
\begin{align}\label{030426_c}
&\lim_{m\to\infty} \int_0^T\int_0^t e^{- c s} \erwb \langle \plap(\uhnr(s)), \uhnr(s) \rangle_{V^*,V} \erwe\,ds\,dt\nonumber\\
=\ &\int_0^T\int_0^t e^{- c s} \erwb \langle \plap(u(s)), u(s) \rangle_{V^*,V}\erwe\,ds\,dt.
\end{align}
Plugging \eqref{030426_c} in \eqref{020426_d}, one gets that
\begin{equation}\label{120526_a}
\int_0^T\int_0^t e^{- c s}\erwb \int_D \big(\psi_2(s,x)-\psi(s,x)u(s,x)\big)\,dx\erwe\,ds\,dt\leq 0.
\end{equation}
Thanks to Lemma \ref{CVSppapn} and Remark \ref{cvtpsi2}, we can affirm that \Pas \ and almost everywhere in $(0,T)\times D$, $0\leq u\leq 1$ and $\psi=\psi_1+\psi_2$ with $\psi_1\leq 0$ and $\psi_2\geq 0$, and then $\psi_2-\psi u\geq 0$. Using this last information in \eqref{120526_a} and following \cite[Proposition 5.12]{BSVZ2025}, one can conclude directly that $\psi\in \partial I_{[0,1]}(u)$.\\
Now it remains to identify $\mathcal{G}$ and $\mathcal{B}$, which are the weak limits of the sequences $(G(\uhnr))_m$ and $(\beta(\uhnr))_m$, respectively. 
To do so, we will prove that the convergence of $(\uhnr)_m$ towards $u$ holds strongly in $L^p(\Omega_T;V)$ and afterward take advantage of the Lipschitz continuity of $G$ and $\beta$ to conclude.
Owing to the fundamental algebraic inequality from \cite[Lemma 2.1., p.107]{Benedetto1989} which states that for any $p \in [2,\infty)$ and any $d\in \na$, there exists a constant $C_p\in\R_+^\star$ that depends only on $p$, such that
\begin{align}\label{130426_a}
&\forall \zeta, \eta \in \R^d,\quad C_p |\zeta - \eta|^p\leq \left( |\zeta|^{p-2}\zeta - |\eta|^{p-2}\eta \right) \cdot (\zeta - \eta),
\end{align}
one is able to write the following equalities for almost all $(\omega,s,x)$ in $\Omega\times (0,T)\times D$ (by omitting the random variable $\omega$):
\begin{align*}
&C_p\big(|\nabla \uhnr(s,x) - \nabla u(s,x)|^p + |\uhnr(s,x) - u(s,x)|^p\big)\smallskip\\
\leq\ &
 \big(|\nabla \uhnr(s,x)|^{p-2}\nabla \uhnr(s,x) - |\nabla u(s,x)|^{p-2}\nabla u(s,x)\big)
\cdot \big( \nabla\uhnr(s,x) -\nabla u(s,x) \big)\\
&+ \big( |\uhnr(s,x)|^{p-2}\uhnr(s,x) - |u(s,x)|^{p-2}u(s,x)\big)
\big( \uhnr(s,x) - u(s,x) \big)\smallskip\\
=\ &|\nabla \uhnr(s,x)|^{p-2}\nabla \uhnr(s,x) \cdot \nabla \uhnr(s,x)
+ |\uhnr(s,x)|^{p-2}\uhnr(s,x) \uhnr(s,x)\\
&- |\nabla \uhnr(s,x)|^{p-2}\nabla \uhnr(s,x) \cdot \nabla u(s,x)
- |\uhnr(s,x)|^{p-2}\uhnr(s,x) u(s,x)\\
&- |\nabla u(s,x)|^{p-2}\nabla u(s,x)\cdot \left( \nabla\uhnr(s,x) - \nabla u(s,x) \right)\\
&- |u(s,x)|^{p-2}u(s,x) \left( \uhnr(s,x) - u(s,x) \right).
\end{align*}
Using this, we arrive at the following estimate for the superior limit of the $L^p(\Omega_T;V)$ norm of $(\uhnr-u)_m$:
\begin{align*}
& \limsup_{m\to\infty}C_p\int_0^T\int_0^t e^{- c s} \erwb \int_D\big(|\nabla \uhnr(s,x) - \nabla u(s,x)|^p + |\uhnr(s,x) - u(s,x)|^p\big)\,dx \erwe\,ds\,dt\nonumber\\
\leq \ &\limsup_{m\to\infty}\int_0^T\int_0^t e^{- c s} \erwb \langle \plap \uhnr(s), \uhnr(s) \rangle_{V^*,V} \erwe\,ds\,dt\nonumber\\
&-\lim_{m\to\infty}\int_0^T\int_0^t e^{- c s} \erwb \int_D |\nabla \uhnr(s,x)|^{p-2}\nabla \uhnr(s,x) \cdot \nabla u(s,x)\,dx \erwe\,ds\,dt\nonumber\\
&-\lim_{m\to\infty}\int_0^T\int_0^t e^{- c s} \erwb \int_D |\uhnr(s,x)|^{p-2}\uhnr(s,x) u(s,x)\,dx \erwe\,ds\,dt\nonumber\\
&-\lim_{m\to\infty}\int_0^T\int_0^t e^{- c s} \erwb \int_D |\nabla u(s,x)|^{p-2}\nabla u(s,x)\cdot \left( \nabla\uhnr(s,x) - \nabla u(s,x) \right)\,dx \erwe\,ds\,dt\nonumber\\
&-\lim_{m\to\infty}\int_0^T\int_0^t e^{- c s} \erwb \int_D |u(s,x)|^{p-2}u(s,x) \left( \uhnr(s,x) - u(s,x) \right)\,dx\erwe\,ds\,dt.
\end{align*}
Using the weak convergences of $(|\nabla \uhnr|^{p-2} \nabla \uhnr)_m$ towards $\vec{z}$ in $L^{p'}(\Omega_T;(L^{p'}(D))^d)$, of $(|\uhnr|^{p-2} \uhnr)_m$ towards $y$ in $L^{p'}(\Omega_T;L^{p'}(D))$ and of $(\uhnr)_m$ towards $u$ in $L^p(\Omega_T;V)$, as well as inequality \eqref{020426_e}, we arrive at
\begin{align*}\label{030426_b}
& \limsup_{m\to\infty}C_p\int_0^T\int_0^t e^{- cs} \erwb \|\uhnr(s) - u(s)\|_V^p \erwe\,ds\,dt\nonumber\\
\leq \ &\int_0^T \int_0^t e^{- c s} \erwb\langle -\dvz(s)+y(s), u(s) \rangle_{V^*,V}\erwe\,ds\,dt\nonumber\\
&-\int_0^T\int_0^t e^{- c s} \erwb \int_D \vec{z}(s,x)\cdot \nabla u(s,x)\,dx \erwe\,ds\,dt\nonumber\\
&-\int_0^T\int_0^t e^{- c s} \erwb \int_D y(s,x)u(s,x)\,dx\erwe\,ds\,dt.\nonumber
\end{align*}
Recalling the meaning of the notation $\dvz$ given by \eqref{HWN} and identifying the integral $\int_D y(s,x)u(s,x)dx$ and the duality bracket $\langle y(s),u(s)\rangle_{V^*,V}$, one deduces that the right hand side term of this last inequality is zero, which proves the strong convergence of $(\uhnr)_m$ towards $u$ in $L^p(\Omega_T;V)$.
Hence, one may conclude, thanks to \eqref{G3_inequality_G_q2} and Assumption $(H_2)$, that $\mathcal{G}=G(u)$ and $\mathcal{B}=\beta(u)$.\\
Finally, going back to \eqref{31032026_a} with all this new information available, one gets the existence of a pair $(u,\psi)$ satisfying
\begin{equation*}
0\leq u\leq 1\text{ and }\psi\in \partial I_{[0,1]}(u),
\end{equation*}
\Pas \ and almost everywhere in $(0,T)\times D$, such that
\begin{equation*}\label{030426_d}
u(t)-u_0+\int_0^t\left(\plap u(s)+\psi(s)\right)\,ds
=\int_0^t G(u(s))\,dW(s)+\int_0^t\left(\beta(u(s))+f(s)\right)\,ds,
\end{equation*}
holds \Pas \ and for any $t\in [0,T]$, in $L^2(D)$.
\end{proof}
\begin{rem}
As mentioned by \cite{BSVZ2025}, thanks to \cite[Corollary 1.2.23, p.25]{HNVW16} with $(S,\mathcal{A},\mu)=(\Omega,\mathcal{F},\mathds{P})$, $(T,\mathcal{B},\nu)=((0,T),\mathcal{B}(0,T), D)$, $X=L^2(D)$, the set $L^2(\Omega;L^2(0,T;L^2(D)))$ is isometrically isomorphic to $L^2(0,T;L^2(\Omega;L^2(D)))$, which allows us to state that $(\uhnr)_m$ also converges strongly towards $u$ in $L^2(0,T;L^2(\Omega;L^2(D)))$.
Combining this last information with the boundedness of $(\uhnr)_m$ in $L^\infty(0,T;L^2(\Omega;L^2(D)))$ (given by Lemma \ref{210611_lem01}) one is able to state that the convergence of $(\uhnr)_m$ towards $u$ also holds strongly in $L^q(0,T;L^2(\Omega;L^2(D)))$ for any finite $q \geq 1$ thanks to Vitali's theorem \cite[Corollaire 1.3.3]{D}.
\end{rem}
\subsection{Uniqueness result}
\begin{prop}
Problem \eqref{equation} admits at most one solution in the sense of Definition \ref{solution}.
\end{prop}
\begin{proof}
Assume that there exist two pairs of solutions $(u, \psi)$ and $(\hat u, \hat\psi)$ for Problem \eqref{equation} in the sense of Definition \ref{solution}. By applying It\^o's formula \cite[Theorem 4.2.5]{LR} for Hilbert-Schmidt processes to the functional $(t,v) \mapsto e^{- \alpha t} \|v\|_{L^2(D)}^2$ defined (for a fixed $\alpha>0$) on $[0,T]\times L^2(D)$ and to the stochastic process $u-\hat u$, for any $t\in [0,T]$ we arrive at
\begin{align*}
&e^{- \alpha t}\erwb \| (u-\hat u )(t) \|_{L^2(D)}^2\erwe + 2\erwb \int_0^t e^{- \alpha s} \langle \plap u(s)-\plap\hat u(s), u(s)-\hat u(s) \rangle_{V^*,V}\,ds\erwe\nonumber\\
&+2 \erwb\int_0^t e^{- \alpha s} (\psi(s)-\hat \psi(s), u(s)-\hat u(s))_{L^2(D)}\,ds\erwe\nonumber\\
=\ & - \alpha\erwb \int_0^t e^{- \alpha s} \| u(s)-\hat u(s) \|_{L^2(D)}^2\,ds\erwe\\
&+ 2 \erwb\int_0^t e^{- \alpha s} (\beta(u(s))-\beta(\hat u(s)), u(s)-\hat u(s))_{L^2(D)}\,ds\erwe\nonumber\\
&+ \erwb\int_0^t e^{- \alpha s} \| G(u(s))-G(\hat u(s)) \|_{\HS}^2\,ds\erwe.\nonumber
\end{align*}
Using the algebraic inequality \eqref{130426_a} twice together with the monotonicity of the subdifferential operator $\partial I_{[0,1]}$, we obtain \Pas \ and almost everywhere in $(0,T)$,
\begin{equation*}
\langle \plap u-\plap\hat u, u-\hat u \rangle_{V^*,V}+(\psi-\hat \psi, u-\hat u)_{L^2(D)}\geq 0,
\end{equation*}
and by choosing the constant $\alpha$ such that $\alpha > L_{G,2}^2 + 2 L_\beta$,
one gets that $u=\hat u$ in $L^2(\Omega;\mathscr{C}(0,T;L^2(D)))$. Finally, since the pairs $(u, \psi)$ and $(u, \hat \psi)$ satisfy \Pas \ and for all $t\in [0,T]$ the equality \eqref{130426_b} from Definition \ref{solution}, one arrives at $\psi=\hat \psi$ in $L^2(\Omega_T;L^2(D))$, and thus the uniqueness result holds.
\end{proof}
\textbf{Acknowledgments} This study received support from the German Research Foundation project (ZI 1542/3-1) and from diverse Procope programs: Procope Mobility Program (DEU-22-0004 LG1) and Project-Related Personal Exchange France-Germany (49368YE).
\bibliographystyle{plain}
\bibliography{main.bib}
\end{document}